\documentclass[reqno,10pt]{amsart}
\usepackage[english]{babel}
\usepackage{amsmath,amsfonts,amsthm,amssymb,amsbsy,upref,color,graphicx,hyperref,enumerate,comment}
\usepackage[a4paper,margin=2.84truecm]{geometry}
\usepackage{bbm}
\usepackage{soul}
\usepackage[hyperpageref]{backref}
\usepackage{xinttools} % for the \xintFor***
\usepackage{tikz}
\usetikzlibrary{positioning,calc,shapes.geometric,angles,quotes}
\usepackage{pgfplots}
\usepackage{ulem}
\pgfplotsset{compat = newest}

\DeclareMathOperator{\diam}{diam}

\def\ds{\displaystyle}
\def\eps{{\varepsilon}}

\def\O{\Omega}
\def\Om{\Omega}
\def\N{\mathbb{N}}
\def\R{\mathbb{R}}

\def\F{\mathcal{F}}

\def\la{\lambda}
\def\eps{\varepsilon}
\def\vps{\varepsilon}
\def\pa{\partial}
\def\sub{\subseteq}
\def\sq{\subseteq}
\def\sm{\setminus}

\newcommand{\be}{\begin{equation}}
\newcommand{\ee}{\end{equation}}
\newcommand{\bib}[4]{\bibitem{#1}{\sc#2: }{\it#3. }{#4.}}

\def\avint{\mathop{\rlap{\hskip2.5pt---}\int}\nolimits}

\numberwithin{equation}{section}
\theoremstyle{plain}

\newtheorem{theo}{Theorem}[section]
\newtheorem{lemm}[theo]{Lemma}
\newtheorem{coro}[theo]{Corollary}
\newtheorem{prop}[theo]{Proposition}

\theoremstyle{definition}
\newtheorem{rem}[theo]{Remark}

\title[ ]{  Mean-to-max ratio of the torsion function and honeycomb structures}
\author[L. Briani]{Luca Briani}

\author[D. Bucur]{Dorin Bucur}

\date{}

\begin{document}

\begin{abstract}
In this paper we study extremal behaviors of the mean to max ratio of the $p$-torsion function with respect to the geometry of the domain. For $p$ larger than the dimension of the space $N$, we prove that the upper bound is uniformly below $1$, contrary to the case $p \in (1,N]$. For  $p=+\infty$,  in two dimensions, we prove that the upper bound is asymptotically attained by a disc from which is removed a network of points consisting on the vertices of a tiling of the plane with regular hexagons  of vanishing size. 
\end{abstract}

\maketitle



\tableofcontents

%%%%%%%%%%%%%%%%%%%%%%%%%%%%%%%%%%%%%%%%%%%%%%%%%%
\section{Introduction}
Let $1<p<+\infty$, $N$ any positive integer and $\O\subset\R^N$ be any nonempty open set with finite Lebesgue measure, $0<|\O|<+\infty$. 
%It is well known that for such a set the space $W^{1,p}_0(\O)$ continuously embeds in $L^{1}(\O)$. 
%We define the \textit{p-torsional rigidity} of the set $\O$ to be the following constant:
%$$
%T_p(\O)=\max\left\{\left(\int_\O |w(x)|dx\right)^{p}\left(\int_\O |\nabla w(x)|^p dx\right)^{-1}: w\in W^{1,p}_0(\O)\setminus\{0\}\right\}.
%$$
We denote by $w_{p,\O}$ the unique solution to the following boundary value problem
\be\label{eq.pdetor}
\begin{cases}
-\Delta_p w=1&\hbox{in }\O,\\
w\in W^{1,p}_0(\O),
\end{cases}
\ee
where $\Delta_p$ stands for the usual $p-$Laplace operator, defined by 
$$\Delta_p w=\mathrm{div}(|\nabla w|^{p-2}\nabla w).
$$
Here, equation \eqref{eq.pdetor} has to be intended in the usual weak sense, that is
$$
\int_{\O}|\nabla w(x)|^{p-2}\nabla w(x)\nabla \psi(x)dx=\int_\O \psi(x)dx,\quad \hbox{for every } \psi \in W^{1,p}_0(\O).
$$
When $p=2$, problem \eqref{eq.pdetor} is usually known as the torsion problem for $\O$ and the corresponding function $w_{2,\O}$ as the torsion function of $\O$. We adopt such a denomination for $p\neq 2$ as well.

 The  torsion function is being studied for many years. The first results relating the geometry of the domain $\Om$ to qualitative properties of $w_{2, \Om}$  are due to Saint-Venant. Indeed, more than 170 years ago he studied mechanical properties of beams with constant cross section in a model corresponding to the Laplace operator, $p=2$. More  sophisticated phenomena involving plastic deformation under a power creep-law appeal to the $p$-torsion function   for large $p$ (see \cite{BaDiMA}).  The extremal case, $p=+\infty$, pops up in  economical problems, for instance in the location of production centers (see \cite{MoBo}). The latter situation corresponds to a  purely geometrical situation as the torsion function equals formally the distance to the boundary.  It is also important to notice that, from a technical point of view, the  torsion function completely controls the $\Gamma $ convergence of the energies $\Om \to \int_{\R^N} |\nabla u|^p dx + \infty _{[W^{1,p}_0(\Om)]^c}$ over $L^p(\R^N)$. In particular, for $p=2$, this implies a full control of the spectrum of the Dirichlet Laplacian for variations of  the geometric domain $\Om$. In very recent applications,  the geometry of the torsion function is intensively studied in order to understand localization properties of the high order eigenfunctions (see \cite{ADFJM19}).

This paper is devoted to the study of some extremal behavior of the $p$-torsion function and, in particular, to answer some questions left open in \cite{bribu}. We consider the \textit{mean-to-max} ratio of the $p$-torsion function
\be\label{efficiency}
\Phi_p(\O)=\frac{1}{\|w_{p,\O}\|_{L^{\infty}(\O)}}\avint_\O w_{p,\O}(x)dx,
\ee
where $\|\cdot\|_{L^{s}(\O)}$, stands for the usual norm of $L^s(\O)$, with $1\le s\le +\infty$ and the symbol $\avint_\O$ denotes the mean.
According to the terminology used in the literature, we refer to the quantity defined by \eqref{efficiency} as the \textit{efficiency} of the function $w_{p,\O}$. In the case of the first Dirichlet eigenfunction, the study of the efficiency dates back to Payne and Stackgold \cite{PaSt} where the authors focused on  a problem related to the design of a nuclear reactor with flat neutron profile. More recently, efficiency has been considered, for different functions, in \cite{bebuka}, \cite{BeDe}, \cite{beka}.

Clearly the value of $\Phi_p(\O)$ cannot be larger than $1$, and an important question is to understand whether or not this limit value can be asymptotically attained by a sequence of geometries. In  \cite{HeLuPi}, the case $p=2$ is considered and it is proved that 
\be
\begin{split}
\label{resulthelupiSUPINF}
&\sup\{\Phi_2(\O):\O\sub\R^N, \hbox{ open set with }0<|\O|<+\infty\}=1,
\end{split}
\ee
the value $1$ being asymptotically attained  by a homogenizing sequence \`a la Cioranescu-Murat. The motivation to study this problem in \cite{HeLuPi} was related to an inequality of Payne involving the torsional rigidity and the first Dirichlet eigenvalue in a competitive way. 
%while $\inf\{\Phi_2(\O):\O\sub\R^N, \hbox{ open set with }0<|\O|<+\infty\}=0$. 

Our first purpose is to analyse  the supremum of $\Phi_p(\O)$ for arbitrary $p\neq 2$.
While for $1<p\le N$ the homogenization technique  used in \cite{HeLuPi} still works with no important extra difficulties, the case $p>N$ is quite challenging. Indeed, when $p>N$, the homogenization phenomenon does not occur anymore since points have positive $p$-capacity. A different behavior may be expected, namely that  the supremum  is strictly less than $1$ (see  \cite[Open Problems 2 and 4]{bribu}). Our first result contains a proof of this assertion, for $N<p<+\infty$. The main difficulty is that a maximizing sequence would $\gamma_p$-converge to the empty set, the absence of a nontrivial limit making the proof  quite technical.

 We shall also consider the case $p=+\infty$, which is particular. The  pointwise limit behavior of the torsion function $w_{p,\O}$ when $p\to +\infty$ is well known (see \cite{BaDiMA} and \cite{Ka90}). One has
$$
\lim_{p\to +\infty}w_{p,\O}(x)=d(x,\O^c),
$$
where $d(x,\O^c)$ denotes the usual distance function from the point $x$ to the set $\O^c$, that is
$$
d(x,\O^c)=\inf\{|x-y|\ : y\in \O^c\}.
$$

 When $p\to+\infty$, the efficiency $\Phi_{p}(\O)$ of the $p$-torsion function converges to 
$$\Phi_{\infty}(\O)=\frac{1}{\|d(\cdot,\O^c)\|_{L^{\infty}(\O)}}\avint_{\O} d(x,\O^c).$$
Note that  $\|d(\cdot,\O^c)\|_{L^{\infty}(\O)}= \rho(\O)$, where  $\rho(\O)$ stands for the  inradius of $\O$, that is the radius of the largest ball which can be inscribed in $\O$. 
For convex sets, this quantity has been investigated in \cite{bribu}  and it has been proved  that
$$
\frac{1}{N+1}\le \Phi_{\infty}(\O)\le \frac{1}{2}.
$$
In this case both the constants are sharp: the left-hand side equality is attained by any ball, while the right-hand one by a sequence of thinning rectangles type domains.

In this paper we remove the convexity constraint and look for the upper bound of $\Phi_{\infty}$ among arbitrary open sets. A first result is that the supremum of  $\Phi_{\infty}$ is less than $1$. The proof is somehow more direct than the one for $\Phi_p$ (with $N<p<+\infty$), but still quite technical, and is not a consequence of it. For this reason we give its main lines. 

The two dimensional case is particularly interesting since the supremum of  $\Phi_{\infty}$ is conjectured to be precisely equal to $\frac13 + \frac{\log 3}{4}$, being asymptotically attained by a  sequence of sets with a honeycomb structure of boundary points. Such a sequence can be obtained, for instance,  by removing from a fixed open set $\Omega$  a network of points given by  the vertices of a tiling of the plane with regular hexagons  of vanishing size. This behavior was  conjectured in \cite[Open Problem 4] {bribu}. The second main result of our paper is the proof of this assertion.

Hexagonal tilings are recurrent structures which come naturally in a series of optimal  partition/location problems. Beyond the celebrated papers of Hale  \cite{Ha} and  Morgan-Bolton  \cite{MoBo}, we refer the reader to \cite{BuFr18}, \cite{BuFr19}. There are not many strategies to prove that optimal structures asymptotically behave like a honeycomb geometry. Hales proves the honeycomb conjecture by finding an ad hoc hexagonal isoperimetric inequality. The few other cases, in which a unified strategy could be followed, correspond to a decomposition of the domain in convex polygons in association with a suitable   polygonal isoperimetric inquality. This strategy was followed by Morgan and Bolton in the economical location problem  \cite{MoBo}, but goes back to  Fejes T\'oth \cite{FT} who first proved the honeycomb conjecture for convex cells.

In our analysis, the  unified strategy used in \cite{MoBo,FT, BuFr18} consisting  in partitioning  the set in a union of convex polygons in association with polygonal isoperimetric inequalities, does not apply. This is due to   the non-local nature of the inradius which prevents us to write down, in any useful way, optimality conditions. Our proof exploits instead an ad hoc isoperimetric inequality in relationship with the Delaunay triangulation associated to a fictious cloud of points spread inside the open set $\Omega$.

An intuitive conclusion of our results is not only that  the supremum values of the efficiency of the torsion function \eqref{efficiency} are of different nature depending on $p$ (below or above $N$) but also the asymptotical maximization structures might behave differently, being much more rigid for $p>N$ than for $p\le N$.
 For instance, in the two dimensional case one may expect that for $p>2$ the maximizing structures are {\it only} the hexagonal ones, while for $p \le 2$ some freedom is left with maximizing structures of different geometry. Any kind of periodical structures in the Cioranescu-Murat analysis which are $\gamma$-converging to a constant multiple of the Lebesgue measure lead to the supremum value $1$.

Before presenting the plan of the paper and state the main results, let us comment briefly on the comparison between the maximization of the efficiency $\Phi_{\infty}$ and the economical location problem of Morgan and Bolton, in the context of the torsion problem. 
We start by recalling a (still open) conjecture of Buttazzo, Santambrogio and  Varchon from 2006 (see \cite{BSV06})  on the optimal compliance location problem. Let $\Om\sq \R^2$ be a bounded open set    and $c\ge  0$. For every $n \in \N$ and for $r=cn^{-1/2}$ one solves the problem 
\begin{equation}\label{brbu01}
\min \left \{\int_\Om w_{2,\Om\sm \cup_{i=1}^n \overline B_r(x_i)}(x) dx : x_1, \dots, x_n \in \R^N\right \}.
\end{equation}
 For $c >0$, the optimal distributions of the balls is conjectured to be     asymptotically given by a  tiling of the set $\Om$ with regular hexagons (the centers of the balls being the vertices of the hexagons, see \cite{BSV06}) when $n \to +\infty$. In fact, this problem naturally extends to all $p>1$, including $p=+\infty$. This latter case, in which  $c=0$ is a nontrivial choice, leads precisely to the 
 Morgan-Bolton  economical location problem.    However, the compliance conjecture  is still open for  any $p<+\infty$. 

Let us turn to the more subtile model which  involves not only the $L^1$ norm of the state function, but also its $L^\infty$-norm, in a competing way. The counter part of  problem \eqref{brbu01}  reads 
\begin{equation}\label{brbu02}
\max \left \{   \frac {\ds \int_\Om w_{2,\Om\sm \cup_{i=1}^n \overline B_r(x_i)}(x) dx} {\| w_{2,\Om\sm \cup_{i=1}^n \overline B_r(x_i)}\|_\infty} : x_1, \dots, x_n \in \R^N\right \}.
\end{equation}
Extending the problem above to $p=+\infty$, the choice $c=0$ leads precisely to maximization of $\Phi_{\infty}$.  
A possible interpretation as an economical location problem is that the maximization of  $\Phi_{\infty}$ corresponds to the modeling   of repulsive spots for which the average distance has, contrary to problem \eqref{brbu01}, to be large, while a too large individual distance is not acceptable (e.g. a territorial distribution of commercial malls or hospital centers). 

Looking to \eqref{brbu01} and \eqref{brbu02} in the limit case $p=+\infty$, the fact that both solutions are honeycombs may be intriguing, as one is a minimization and the other is maximization of ``almost" the same quantity. Of course, the difference comes from  the presence in  \eqref{brbu02} of the $L^\infty $-norm which,  alone, it is also expected to be minimal on honeycomb structures, naturally introducing a competition between the $L^1$ and the $L^\infty$ norms. The main consequence of this competition  makes the unified strategy  of \cite{MoBo,FT, BuFr18}  impossible to follow.

\medskip

\noindent {\bf Plan of the paper.}
This paper is organized as follows. In Section \ref{spre} we recall some preliminary results and briefly describe the minimization problem of $\Phi_p$. For the sake of completeness, we also give a short proof of the upper bound   \eqref{resulthelupiSUPINF}  in the case $1<p\le N$, extending the result of  \cite{HeLuPi}.

 Section \ref{spun} contains our first main result. 
 \begin{theo}\label{brbu04}
Let $N<p\le +\infty$. Then 
$$\sup\left\{\Phi_p(\O): \ \O\subset\R^N,\hbox{ open set with }  0<|\O|<+\infty\right\}<1.$$
\end{theo}

In fact, for $N<p<+\infty$, we shall prove a stronger version of this result, showing that
$$\sup \left \{\left(\frac{1}{|\O|}\int_{\O} w_{p,\O}(x)\right)\left(\frac{1}{|\O|}\int_{\O}w^p_{p,\O}(x)\right)^{-1/p}\right \}<1.$$

The proof of our second result is  given in Section \ref{sdist}. 
\begin{theo}\label{brbu05}
We have
$$\sup\left\{\Phi_\infty(\O): \ \O\subset\R^2,\hbox{ open set with }  0<|\O|<+\infty\right\}=\frac 1 3+\frac {\ln(3)} {4}.$$
Denoting $H_1$ the set of the vertices of a tiling of the plane by regular hexagons of area $1$, the supremum above is attained by any sequence   $(\Omega\setminus H_\vps)_{\vps}$, where $\Omega$ is a smooth bounded open set, $H_\vps=\vps H_1$ and $\vps \to 0_+$.
\end{theo}
We finish with some remarks and applications in Section \ref{brbu10}, in particular we give a positive answer to Open problem $2$ of \cite{bribu} concerning the shape optimization of a functional involving a competition between the first eigenvalue of the $p$-Laplacian and the $p$-torsion energy.

\section{Preliminaries and the  sub-dimensional case}\label{spre}
Let $1<p<+\infty$  and 
$\mathfrak{F}_p:W^{1,p}_0(\O)\rightarrow \R$
 be the strictly convex functional defined by
\be\label{functional}
\mathfrak{F}_p(u)=\frac{1}{p}\int_\O |\nabla w(x)|^pdx-\int_\O w(x)dx. 
\ee
A simple application of the direct method of calculus of variation proves that there exists a unique minimizer for $\mathfrak{F}_p$. Such a minimizer wealky solves the boundary problem \eqref{eq.pdetor} and coincide then with $w_{p,\O}$. 
Being $\mathfrak{F}_p(v)\ge \mathfrak{F}_p(|v|)$ we have that $w_{p,\O}\ge 0$ in $\O$.
By testing \eqref{eq.pdetor} with $w_{p,\O}$ itself, we deduce that
\be\label{torsionprop}
\int_{\O}|\nabla w_{p,\O}(x)|^pdx=\int_\O w_{p,\O}(x)dx
\ee 
When necessary, we identify $w_{p,\O}$ with its extension in $W^{1,p}(\R^N)$ obtained by  defining $w_{p,\O}$ to be $0$ in $\R^N\setminus\O$. With such an identification it holds  $-\Delta_p w_{p,\O}\le 1$ weakly in $\R^N$, that is
\be\label{extension}
\int_{\R^N}|\nabla w_{p,\O}(x)|^{p-2}\nabla w_{p,\O}(x)\nabla \phi(x)dx\le \int_{\R^N}\phi(x)dx
\ee
for every $\phi\in W^{1,p}(\R^N)$, $\phi\ge 0$.

An explicit computation of $w_{p,\O}$ is, in general, not available, except for some very specific choice of the domain $\O$. For instance if $\O=B(x_0,r)$, where we denote by $B(x_0,r)$ the open open ball of $\R^N$ centered at $x_0$ and with radius $r>0$, is easy to verify that
\be\label{torsionball}
w_{p,B(x_0,r)}(x)=\frac{r^{p'}-|x-x_0|^{p'}}{p'N^{p'/p}},
\ee
$p'=p/(p-1)$ being the conjugate exponent of $p$.
Also, it is worth recalling that, due to the degeneracy of the operator $\Delta_p$, weak solutions to \eqref{eq.pdetor} not always belong to $C^2(\O)$ as already  \eqref{torsionball} shows. In fact, one can prove that $w_{p,\O}\in C^{1,\alpha}(\O)$ with  regularity holding  up to the boundary if the domain $\O$ is regular enough, see for instance the classical results in \cite{Dibe}.
%We will repetitively using the following comparison principle, see \cite{Sa}.
%\begin{lemm}[Comparison principle]\label{comparison1}
%Let $\O$ be any bounded smooth open set of $\R^N$. Let $u_1,u_2\in W^{1,p}(\O)$ such that
%$-\Delta_p u_1\le -\Delta_p u_2$ weakly in $\O$, that is
%$$
%\int_\O |\nabla u_1(x)|^{p-2}\nabla u_1(x) \nabla \phi(x)dx\le \int_\O |\nabla u_2(x)|^{p-2}\nabla u_2(x) \nabla \phi(x)dx,
%$$
%for all $\phi\in W^{1,p}_0(\O)$, $\phi\ge 0$. Then, if
%$$u_1\le u_2 \hbox{ on } \pa\O,$$ it holds 
%$$u_1\le u_2 \hbox{ on } \O.$$
%\end{lemm}
Using the explicit expression \eqref{torsionball}, the fact that $w_{p,\O}\ge 0$ and the comparison principle we can easily deduce that $w_{p,\O}>0$ in $\O$.
 
The following Caccioppoli-type estimate, see  \cite{bebu} Lemma $10$, can be obtained by considering $\theta\in C^{\infty}_c(\R^N)$  and testing \eqref{extension} with $\theta^pw_{p,\O}\in W^{1,p}(\R^N)$:
\be\label{caccioppoli}
\begin{split}
\int_{\R^N} |\nabla(w_{p,\O}(x)\theta(x))|^pdx &\le c_1\int_{\R^N} w_{p,\O}(x)|\theta(x) |^pdx\\
&+c_2\int_{\R^N} w^p_{p,\O}(x)|\nabla \theta(x)|^p dx.
\end{split}
\ee
Here $c_1$ is any constant with $c_1> 2^{p-1}$ and $c_2=c_2(p,c_1)$, where we use the convention of writing $c=c(s_1,\dots,s_n)$ to denote a constant $c>0$ whose value depends only on some quantities $s_1,\dots,s_n\in \R$.
%$$c_2=c_2(c_1,p)=2^{p-1}+c_1\left(\frac{(p-1)c_1}{c_1-2^{p-1}}\right)^{p-1}.$$

In \cite{bebu} is also noticed (see inequality $(30)$)  that, as a consequence of \eqref{caccioppoli}, one gets the existence of two constants $C_1=C_1(N,p,c_1),C_2=C_2(N,p,c_1)$ such that for every $r>0$ and $x_0\in \R^N$ it holds
\be\label{caccioppoli2}
w_{p,\O}(x_0)\le C_1\left(\avint_{B(x_0,r)}w_{p,\O}^p(x)dx\right)^{1/p}+C_2r^{p'}.
\ee
We recall the well known Morrey inequality, see \cite{EvGa}.
\begin{lemm}[Morrey inequality]\label{morrey}
For each $ N<p<+\infty$ there exists a constant $C=C(N,p)$, such that
$$
|u(y)-u(z)|\le C r^{1-\frac{N}{p}} \left(\int_{B(x_0,r)}|\nabla u(x)|^{p}dx\right)^{1/p}
$$
for all $B(x_0,r)\subset\R^N$, $u\in W^{1,p}(B(x_0,r))$ and for almost every $y,z\in B(x_0,r)$.
\end{lemm}

At last, we notice that the minimization problem for $\Phi_{p}$ is not interesting since it is easy to show that that 
\[
\begin{split}
&\inf\{\Phi_p(\O):\O\sub\R^N, \hbox{ open set with }0<|\O|<+\infty\}=0.
\end{split}
\]
Indeed, let $\O_n$ be the following open set
$$
\O_n= \bigcup_{k=1,\dots, n}B(2k,r_k), \quad \hbox{where }r_k=k^{-1/N},\ \hbox{for }k=1,\dots,n;
$$
clearly we have $\lim_{n\to+\infty}|\O_n|=+\infty$.
Since, by \eqref{torsionball}, it holds
$$w_{p,\O_n}(x)=\sum_{i=1}^{n}w_{p,B(2k,r_k)}(x)=\frac{1}{p'N^{p'/p}}\sum_{k=1}^{n}(r_k^{p'}-|x-2k|^{p'})\chi_{B(2k,r_k)}(x),$$
being $\chi_E$ the characteristic function of the set $E$. In particular $$\|w_{p,\O_n}\|_{L^{\infty}(\O_n)}=\frac{1}{p'N^{p'/p}}, \quad \lim_{n\to+\infty}\int_{\O_n}w_{p,\O_n}(x)dx<+\infty,$$
which in turns imply $\lim_{n\to+\infty}\Phi_p(\O_n)=0$.

\bigskip

Concerning the maximization problem for $\Phi_p$, as already noticed in the Introduction, it is proved in \cite{HeLuPi} that, for $N\ge 2$, it holds
\be\label{resulthelupi}
\sup\{\Phi_2(\O):\O\sub\R^N, \hbox{ open set, with }0<|\O|<+\infty\}=1.
\ee
The proof is based on the following homogenization procedure for which we refer to \cite{ciomur}: denoting $B_1=B(0,1)$, and given $a>0$, there exists a sequence $(\O_n)_{n\in\N}$ of open subsets of $B_1$, such that the sequence $(w_{2,\O_n})_{n\in\N}\subset W^{1,2}(B_1)$ weakly converges to the solution $w_{2,a}$ of the following boundary problem
$$
\begin{cases}
-\Delta w+aw=1, & \hbox{in }B_1;\\
w\in W^{1,2}_0(B_1),
\end{cases}
$$ 
as usual intended in the weak sense.

Once showed that
\be\label{lemmahelupi1}
\lim_{n\to+\infty}\|w_{2,\O_n^a}\|_{L^{\infty}(\O_n^a)}=\|w_{2,a}\|_{L^{\infty}(B_1)},
\ee
\be\label{lemmahelupi2}
aw_{2,a}\rightharpoonup 1\hbox{ in } L^2(B_1)\ \hbox{as }a\to +\infty, \quad  0\le aw_{2,a}(x)\le 1 \hbox{ a.e. in }B_1. 
\ee
the identity \eqref{resulthelupi} is achieved in \cite{HeLuPi}, using \eqref{lemmahelupi1} and \eqref{lemmahelupi2},  through the following inequalities
\be\label{proofhelupi}
\begin{split}
&1\ge \sup\Phi_2(\O)\ge \sup_{a>0}\left(\lim_{n\to+\infty}\Phi_2(\O^a_{n})\right)\ge \sup_{a>0} \frac{\int_{B_1} w_{2,a}(x)dx}{|B_1|\| w_{2,a}\|_{L^{\infty}(B_1)}}\\
&\ge \lim_{a\to+\infty}\frac{\int_{B_1} w_{2,a}(x)dx}{|B_1|\| w_{2,a}\|_{L^{\infty}(B_1)}}=1.
\end{split}
\ee

The whole argument can be basically repeated whenever $1< p\le N$. For completeness, we shortly describe here the technical points needed to prove the analogues of \eqref{lemmahelupi1} and \eqref{lemmahelupi2} in this setting.
First we notice that for every $\mu>0$ there exists a sequence $(\O_{n})_{n\in\N}$ of open subsets of $B_1$ such that the sequence $(w_{p,\O_{n}})_{n\in\N}$ weakly converge in $W^{1,p}(B_1)$ to the solution $w_{p,\mu}$ of the following boundary value problem
\be\label{eq.ciomu}
\begin{cases}
-\Delta_p w+\mu |w|^{p-2}w=1, & \hbox{in }B_1;\\
w\in W^{1,p}_0(B_1).
\end{cases}
\ee
For this result we refer to \cite{LaPi}, see also \cite{bubu05}, \cite{He} and references therein. Notice that here is precisely where the hypothesis $p\le N$ is needed. Then we have the following lemmas.

\begin{lemm}\label{lem.ntoinf}
Let $(\O_n)_{n\in\N}$ be a sequence of open subsets of $B_1$. Then, if $w_{p,\O_n}$ weakly converges to $v\in W^{1,p}_0(B_1)$, we have also
$$
\lim_{n\to+\infty}\|w_{p,\O_n}\|_{L^{\infty}(\O_n)}=\|v\|_{L^{\infty}(B_1)}
$$
\end{lemm}
\begin{proof}
The pointwise convergence of $w_{p,\O_n}$ implies $$\liminf_{n\to+\infty}\|w_{p,\O_n}\|_{L^{\infty}(\O_n)}\ge\|v\|_{L^{\infty}(B_1)}.$$ 
We suppose without loss of generality that
$$
\liminf_{n\to\infty}\|w_{p,\O_n}\|_{L^{\infty}(\O_n)}=\lim_{n\to+\infty}\|w_{p,\O_n}\|_{L^{\infty}(\O_n)},
$$
and we assume by contradiction that there exist $b_1<b_2$ such that
\be\label{contradiction}
\lim_{n\to+\infty}\|w_{p,\O_n}\|_{L^{\infty}(\O_n)}>b_2>b_1>\|v\|_{L^{\infty}(B_1)}.
\ee
We denote by $A_n=\{x\in\O_n: \ w_{p,\O_n}(x)>b_1\}$. Since for $n$ large enough $A_n$ is nonempty, we can select $x_n\in A_n$ to be such that
$$
w_{p,\O_n}(x_n)=\|w_{p,\O_n}\|_{L^{\infty}(\O_n)}>b_2.
$$
Notice that we have
$$
(w_{p,\O_n}-b_1)^+=w_{p,A_n}.
$$
where we use the usual notation $g^+=\max\{g,0\}$.
Hence by \eqref{caccioppoli2} we get, 
$$
b_2-b_1<(w_{p,\O_n}-b_1)^+(x_n)\le C_1\left(\avint_{B(x_n,r)}\left((w_{p,\O_n}-b_1)^+\right)^p(x)dx\right)^{1/p}+C_2r^{\frac{p}{p-1}},
$$
for some universal and fixed constants $C_1$ and $C_2$.
Let $r=r(b_1,b_2,p,N)$ be such that 
$$
C_2r^{p/(p-1)}\le (b_1-b_2)/2
$$
We deduce 
\be\label{cosequence}
0<\frac{b_2-b_1}{2}\le C_1\left(\avint_{B(x_n,r)}\left((w_{p,\O_n}-b_1)^+\right)^p(x)dx\right)^{1/p}.
\ee
However since $w_{p,\O_n}$ converges to $v$ in $L^p(B_1)$ and $b_1>v$, the right hand side of \eqref{cosequence} converges to zero when $n\to+\infty$, which is a contradiction. Hence \eqref{contradiction} cannot hold, and this proves the lemma.
\end{proof}

\begin{lemm}\label{lem.atoinf}
Let $a>0$ and  $w_{p,a^{p-1}}$ be the solution to \eqref{eq.ciomu} when $\mu=a^{p-1}$. Then 
$0\le aw_{p,a^{p-1}}(x)\le 1$ almost everywhere in $B_1$. Moreover
the sequence  $(aw_{p,a^{p-1}})_{a>0}$ converge to $1$ almost everywhere as $a\to+\infty$.
% and 
%$\lim_{a\to+\infty}\|aw_{p,a^{p-1}}\|_{L^{\infty}(B_1)}=1$.
\end{lemm}

\begin{proof}

We denote for the sake of simplicity  $\eps=a^{1-p}$ and $v_\eps=aw_{p,a^{p-1}}$.
First, we want to show that, for any $\eps>0$, it holds 
\be\label{bounds}
0\le v_\eps\le 1 \hbox{ almost everywhere in } B_1.
\ee
Clearly, $v_\eps$ satisfies in the weak $W^{1,p}_0(B_1)$ sense the following problem
\be\label{eqsupp}
\begin{cases}
-\eps\Delta_p v_\eps+|v_\eps|^{p-2}v_\eps=1 & \hbox{in }B_1,\\ 
v_\eps\in W^{1,p}_0(B_1).
\end{cases}
\ee
In particular, $v_\eps$ is the unique minimizer of the strictly convex functional
$$
\mathfrak{F}_{p,\eps}(v)=\frac{\eps}{p}\int_{B_1} |\nabla v(x)|^pdx+\frac{1}{p}\int_{B_1} |v(x)|^pdx-\int_{B_1}v(x)dx
$$
defined in $W^{1,p}_0(B_1)$.Being $\mathfrak{F}_{p,\eps}(v)\ge \mathfrak{F}_{p,\eps}(|v|)$, it holds $v_\eps\ge 0$ in $B_1$. 
Moreover, since $(1+x)^{\alpha}\ge 1+\alpha x$ for every $x>0$ and $\alpha>1$, we have
$$
\frac 1 p \int_{\{v>1\}}\left(|v|^{p}-1\right)dx-\int_{\{v>1\}}(v-1) dx\ge 0,
$$
that in turn implies
$\mathfrak{F}_{p,\eps}(v)\ge \mathfrak{F}_{p,\eps}(\min\{v,1\})$. In particular $v_\eps\le 1$ in $B_1$. 

We can easily notice that if $0<\eps_1<\eps_2$ then $v_{\eps_1} \ge v_{\eps_2}$. Indeed, we have that
$$-\eps_1\Delta_p v_{\eps_2}+ v_{\eps_2}^{p-1}= \frac{\eps_1}{\eps_2} +\left (1-\frac{\eps_1}{\eps_2} \right )  v_{\eps_2}^{p-1}\le 1$$
so that by the comparison principle we get $v_{\eps_1} \ge v_{\eps_2}$ a.e. in $B_1$. In particular, denoting $\overline{v} (x)= \sup_\eps v_\eps(x)$, we get $v_\vps \to \overline{v}  $, a.e. pointwise in $B_1$. 
Since $v_\eps\le 1$ almost everywhere in $B_1$ it also holds $\overline{v}\le 1$ a.e. in $B_1$.
Furthermore, being $\mathfrak{F}_{p,\eps}(v_\eps)\le\mathfrak{F}_{p,\eps}(0)= 0$, we have
$$
 \int_{B_1}\eps|\nabla v_\eps(x)|^pdx\le p|B_1|.
$$
%\quad  \int_{B_1}|v_\eps(x)|^p\le p|B_1|.
Thus, for any $\phi\in W^{1,p}_0(B_1)$ we have
$$
\left|\eps  \int_{B_1}|\nabla v_\eps(x)|^{p-2}\nabla v_{\eps}(x)\nabla \phi(x)dx\right|\le \left(\int_{B_1}\eps| \nabla v_\eps(x)|^{p}dx\right)^{1/p'}\left(\int |\nabla \phi(x)|^{p}dx\right)^{1/p}\eps^{1-1/p'}
$$
which implies
\be\label{boundednabla}
\left| \eps  \int_{B_1}|\nabla v_\eps(x)|^{p-2}\nabla v_{\eps}(x)\nabla \phi(x)dx\right|\le \eps^{1-1/p'}(p|B_1|)^{1/p'}\|\nabla \phi\|_{L^{p}(B_1)}.
\ee
%Similarly we have
%\be\label{boundednabla2}
%\left| \int_{B_1}|v_\eps(x)|^{p-2}v_{\eps}(x)\phi(x)dx\right|\le (p|B_1|)^{1/p'}\|\phi\|_{L^{p}(B_1)}.
%\ee
From equation \eqref{eqsupp}, we obtain
$$
\eps\int_{B_1}|\nabla v_{\eps}(x)|^{p-2}\nabla v_{\eps}(x)\nabla \phi(x)dx+\int_{B_1}|v_{\eps}(x)|^{p-2}v_{\eps}(x)\phi(x)dx=\int_{B_1}\phi(x)dx,
$$
and hence, taking also into account \eqref{boundednabla},
$$
\lim_{\eps\to 0}\int_{B_1}v_{\eps}^{p-1}(x)\phi(x)dx=\int_{B_1}\phi(x)dx.
$$
for any $\phi\in W^{1,p}_0(B_1)$.
By possibly passing to a sub-sequence, being $v_\eps\le 1$ and $v_\eps\to \overline{v}$ a.e. as $\eps\to 0$, the latter implies that 
$$
\int \overline{v}^{p-1}(x)\phi(x)dx=\int_{B_1}\phi(x)dx.
$$
so that, by the arbitrariness of $\phi$, $\overline{v}= 1$ almost everywhere which proves the thesis.
\end{proof}

Combining Lemma \ref{lem.ntoinf} and Lemma \ref{lem.atoinf} we can repeat, with minor differences, the inequalities \eqref{proofhelupi}.

\begin{theo}\label{torsionplN}
Let $1< p\le N$, then:
$$\sup\{\Phi_p(\O):\O\sub\R^N, \hbox{ open set, with }0<|\O|<+\infty\}=1.$$
\end{theo}

\begin{proof}
Let $a>0$, and $(\O^a_{n})_{n\in\N}$  a sequence of open subsets of $B_1$ such that the sequence $w_{p,\O^a_n}$ weakly converges in $W^{1,p}_0(B_1)$ to the solution $w_{p,a^{p-1}}$ of the problem \eqref{eq.ciomu} when $\mu=a^{p-1}$. By applying Lemma \ref{lem.ntoinf} and Lemma \ref{lem.atoinf} we get
\[\begin{split}
&1\ge \sup\Phi_p(\O)\ge \sup_{a>0}\left(\lim_{n\to+\infty}\Phi_p(\O^a_{n})\right)\ge \sup_{a>0} \frac{\int_{B_1} w_{p,a^{p-1}}(x)dx}{|B_1|\| w_{p,a^{p-1}}\|_{L^{\infty}(B_1)}}\\
&\ge \lim_{a\to+\infty}\frac{\int_{B_1} aw_{p,a^{p-1}}(x)dx}{|B_1|\| aw_{p,a^{p-1}}\|_{L^{\infty}(B_1)}}	\ge 1,
\end{split}
\]
which proves the theorem.
\end{proof}

\section{Proof of Theorem \ref{brbu04}: the super-dimensional case}\label{spun}
Before discussing the proof of Theorem \ref{brbu04}, let us briefly note what happens if $N=1$. In this case, it is easy to show that for every $1<p<+\infty$ it holds
\be\label{1D}
\sup\left\{\Phi_p(\O): \O\sub\R, \hbox{ open set with }0<|\O|<+\infty\right\}=\frac{p'}{p'+1}.
\ee

Indeed, every $\O\sub\R$ is a disjoint union of open intervals $(\O_i)_{i\in I}$ of length $2r_i$. Using  \eqref{torsionball} with $N=1$, we have
$$
\int_\O w_{p,\O}(x)dx=\frac{2}{p'+1}\sum_{i\in I}r_i^{p'+1}, \quad  \|w_{p,\O}\|_{L^\infty(\O)}= \frac{\max_{i\in I}r_i^{p'}}{p'}.$$
This implies
$$
\Phi_p(\O)=\frac{p'}{p'+1} \left(\frac{\sum_i r_i^{p'+1}}{\max_i r_i^{p'} \sum_i r_i}\right)\le \frac{p'}{p'+1},
$$
with equality achieved when $r_i$ is any constant a value.
Notice that\eqref{1D} already suggests that the upper bound from Theorem \ref{torsionplN} may fail if we remove the assumption $p\le N$.

\bigskip

From now on we suppose $N\ge 2$ and we assume $N<p<+\infty$.

Given $m,r>0$ and $x_0\in \R^N$, we denote by $V_{m, r,x_0}\in W^{1,p}(B(x_0,r))$ a continuous weak solution to the following boundary value problem
\be\label{V}
\begin{cases}
-\Delta_p V=1, & \hbox{in } B(x_0,r),\\
V(x_0)=0,\\
V(x)=m, & \hbox{in }\pa B(x_0,r).
\end{cases}
\ee
whose existence is guaranteed by the following lemma.
\begin{lemm}\label{capacitylemma}
Let $N<p<+\infty$, $m,r>0$ and $x_0\in \R^N$. There exists a continuous weak solution to problem \eqref{V}. In particular there exists $\rho=\rho(m,r,N,p)<r$  such that, for any $x_{0}\in \R^N$ we have
$$
B(x_0,\rho)\sub\{x\in B(x_0,r)\ :\ V_{m,r,x_0}(x)<1/2\}.
$$
\end{lemm}

\begin{proof}
The argument is standard. Let
$$X=\{ u\in C^0(\overline{B(x_0,r)}):\ u(x_0)=0,\ u(x)=m \hbox{ for every } x\in \pa B(x_0,r) \}.$$
Since $p>N$, the space $W^{1,p}(B(x_0,r))$ continuously embeds in $C^0(\overline{B(x_0,r)})$. In particular, $X \cap W^{1,p}(B(x_0,r))$ is convex and closed in $W^{1,p}(B(x_0,r))$. By Lemma \ref{morrey} we easily deduce that
$$
\|u\|_{L^p(B(x_0,r))}\le C\|\nabla u\|_{L^p(B(x_0,r))}, \hbox{for all }u\in X,
$$
for some $C=C(p,N,r)$.
Hence, the functional $\mathfrak{F}_p$, defined by \eqref{functional}, is coercive on $X\cap W^{1,p}(B(x_0,r))$ and there exists $\overline{u}\in X\cap W^{1,p}(B(x_0,r)$ such that $\min_X\mathfrak{F}_p=\mathfrak{F}_p(\overline{u})$.  Such a minimizer solves in the weak sense equation \eqref{V}. The rest of the Lemma is straightforward.
\end{proof}

As we will use it repeatedly, we recall the statement of the classical Vitali's covering theorem, see for instance \cite{EvGa}.

\begin{lemm}[Vitali's Covering Lemma]
Let $\F$ be a family of closed balls of $\R^N$. Suppose that it holds
$$
\sup\{\diam(B): B\in\F\}<+\infty.
$$
Then, there exists a  countable family of disjoint balls $\F'\subset\F$ such that
$$
\bigcup_{B\in\F}B\sub\bigcup_{B\in\F'} 5B.
$$
\end{lemm}

The key step to prove our main result is the following technical lemma.
\begin{lemm}\label{lemmakey}
Let $N<p<+\infty$ and $\O\sub\R^N$ be an open set with finite measure. Then, there exists a constant $C=C(p,N)>0$ depending only on $N$ and $p$, such that one of the following cases occur:
\begin{itemize}
\item[(i)] $\left |\left \{x\in\O :\ w_{p,\O}(x)<\frac{1}{2}\avint_\O w_{p,\O}(x)dx \right \}\right |\ge \frac{|\O|}{C};$
\item[(ii)] $ \left | \left \{ x\in\O :\ w_{p,\O}(x)>\frac{3}{2}\avint_\O w_{p,\O}(x)dx\right \}\right |\ge \frac{|\O|}{C}.$
\end{itemize}
\end{lemm}

\begin{proof}
We denote for the sake of brevity $w_{p,\O}=w$ and we use the following notation:
$$
w_{0}=\frac{1}{|\O|}\int_{\O}w(x)dx,\quad E_1=\{x\ :\ w<w_0/2\}, \quad E_2=\{x\ : \ w>3w_0/2\}.
$$
By possibly rescaling, we can suppose  that $w_0=1$ and when necessary we can extend $w$ to be zero in $\R^N\setminus\O$, in this case we denote the new function again as $w$.
Our goal is to prove that there exists a constant $C=C(N,p)$ such that at least one of the following cases occurs:
\be\label{goal}
|E_1|\ge \frac{|\O|}{C}, \quad |E_2|\ge \frac{|\O|}{C }.
\ee

We start by introducing some notation. We consider the family $(B(x,d(x,\O^c)))_{x\in\O}$ of open balls. which clearly covers $\O$.
A simple application of the classical Vitali covering Lemma  allows us to define a sub-family 
\be\label{vitalicovering}
\F=\{B_i=B(x_i,r_i):i\in I\},
\ee
with  $r_i=d(x_i,\O^c)$ such that
$$
B_i\cap B_j=\emptyset, \hbox{ if }i\neq j,\quad \hbox{ and }\O\subset \bigcup_{ i\in I} 5B_i.
$$
Thanks to the choice of the radii, we can associate at any ball $B_i\in \F$ a point $y_i$ such that $y_i\in \pa B_i\cap\pa \O$. Notice that possibly $y_i=y_j$ for some $i\neq j$.
We further define $R_0>0$ to be such that
$$w_{p,B(x_0,R_0)}(x_0)=2,\quad \hbox{ for every }x_0\in \R^N,$$
that is, taking into account \eqref{torsionball},  $R_0=(2p')^{1/p'}N^{1/p}$.
We divide $\F$ into the two following sub-families of \textit{large} and \textit{small} balls:
$$
\F_{L}=\{B_i\ : B_i\in\F, \hbox{ s.t. } r_i\ge R_0\}, \quad \F_{S}=\{B_i\ : B_i\in\F, \hbox{ s.t. } r_i<R_0\}.
$$
Moreover, taking $m>0$ to be a constant whose value we fix later, we further divide the family $\F_{S}$ into two disjoint sub-families:
$$
\F_{S_1}=\{B_i\in \F_S\ :\max_{B(y_i,R_0)}w\le m\};\quad \F_{S_2}=\{B_i\in\F_S\ \max_{B(y_i,R_0)}w> m\}.$$

Finally, using Lemma \ref{capacitylemma} and \eqref{torsionball} we choose $0<r_0=r_0(p,N)<R_0$ to be such that, for every $x_0\in\R^N$,
\be\label{rzero}
B(x_0,r_0)\sub\left\{x\in B(x_0,R_0)\ :\ w_{p,B(x_0,R_0)}(x)>3/2,\ V_{m,R_0,x_0}(x)<1/2\right\},
\ee
where $V_{m,R_0,x_0}$ is defined through \eqref{V}.

Now, for the sake of clarity, we proceed by dividing the proof into $4$ different steps.

\noindent \textbf{Step 1:} 
We prove that if the family of large balls $\F_L$ covers enough measure of $\O$ the lemma is proved. More precisely, assume
\be\label{Igoodcase}
\sum_{B\in \F_L}|B|\ge \frac{1}{2}\frac{|\O|}{5^N}.
\ee
Then, for every $B\in \F_L$, we consider the family 
$$\{ B(z,R_0):\ z\in B,\ d(z,B^c)\ge R_0\}.$$ 
Again Vitali's Lemma applies and, being $B$ a bounded set, allows to select a finite number $n=n(B)$ of disjoint balls $\{B(z_i,R_0)\}_{i=1,\dots,n}$ such that $B\subset\cup_{i=1}^{n}5B(z_i,R_0)$.
In particular 
\be\label{aiuto1}
|B|\le n 5^N\omega_NR_0^N.
\ee
Then, for any $i=1,\dots n$, the comparison principle gives 
$$
w(x)\ge w_{p, B(z_i,R_0)}(x), \quad \hbox{ in }B(z_i,R_0)
$$
so that
$$\{w_{p,B(z_i,R_0)}>3/2\}\sub\{w>3/2\}\cap B(z_i,R_0).$$
As a consequence of the former inclusion and of \eqref{rzero}, we get
$$
\omega_Nr_0^N\le \left|\{w>3/2\}\cap B(z_i,R_0)\right|.
$$
Since the balls $B(z_i,R_0)$ are pairwise disjoint, the former inclusion gives
$$
n\omega_N r^N_{0}\le \left|\{w>3/2\}\cap B\right|.
$$
Combining the latter inequality and \eqref{aiuto1} we obtain
$$
\left(\frac{r_0}{5R_0}\right)^N|B|\le \left|\{w>3/2\}\cap B\right|.
$$
Therefore, using \eqref{Igoodcase}, we get
$$
|E_2|=\left|\{w>3/2\}\right|\ge \sum_{B\in\F_L}\left|\{w>3/2\}\cap B\right|\ge \left(\frac{r_0}{5R_0}\right)^N\frac{1}{2}\frac{|\O|}{5^N}
$$
which implies \eqref{goal}, and consequently  the thesis.

\noindent \textbf{Step 2: }We prove that if the family of small balls $\F_{S_1}$ covers enough measure of $\O$ the lemma is proved as well. More precisely, assume 
\be\label{IIgoodcase}
\sum_{B\in \F_{S_1}}|B|\ge \frac{1}{4}\frac{|\O|}{5^N}.
\ee
Let $B(x_i,r_i)\in \F_{S_1}$. Since, by the definition of $\F_{S_1}$, it holds
$$
\max_{B(y_i,R_0)} w\le m,
$$
we have $w\le V_{m,R_0,y_i} \ \hbox{in }\pa B(y_i,R_0)$. Moreover
$-\Delta_p w\le -\Delta_p V_{m,R_0,y_i}$ weakly in $B(y_i,R_0)$ so that, by the comparison principle, it holds
$w\le V_{m,R_0,y_i} \hbox{ in } B(y_i,R_0)$,
and thus, by \eqref{rzero},
\be\label{aiuto2}
B(y_i,r_0)\cap B(x_i,r_i)\sub \{w\le 1/2\}\cap B(x_i,r_i)\cap B(y_i,R_0).
\ee
By the fact that the radii of any ball $B(x_i,r_i)$ is controlled by $R_0$ we also deduce that
\be\label{aiuto3}
|B(y_i,r_0)\cap B(x_i,r_i)|\ge c|B(x_i,r_i)|
\ee
for some constant $c=c(r_0,R_0,N)>0$. Indeed, for $r_i\le r_0/2$  we simply have the inclusion $B(x_i,r_i)\subset B(y_i,r_0)\cap B(x_i,r_i)$, while, for $r_i>r_0/2$ we have
\[
\begin{split}
&|B(y_i,r_0)\cap B(x_i,r_i)|\ge |B(x_i,r_i)|\frac{|B(x_i,r_i)\cap B(y_i,r_0)|}{|B(x_i,r_i)|}\\
&\ge |B(x_i,r_i)|\frac{|B(z,r_0/2)\cap B(y_i,r_0)|}{\omega_NR_0^N}=c|B(x_i,r_i)|,
\end{split}
\]
where $z$ belongs to the segment joining $y_i$ and $x_i$ and is such that $|z-y_i|=r_0/2$.
Now, combining \eqref{aiuto3} and \eqref{aiuto2}, we obtain
$$
c\omega_Nr_i^N\le |\{w\le 1/2\}\cap B(x_i,r_i)\cap B(y_i,R_0)|.
$$
Finally, taking the sum over $B(x_i,r_i)\in \F_{S_1}$ and applying \eqref{IIgoodcase} we obtain
$$
\frac c 4\frac{|\O|}{5^N}\le c\sum_{B\in\F_{S_1}}|B|\le 	\left|\{w\le 1/2\}\right|.
$$
This, again, implies \eqref{goal}, so that the lemma is proved also if \eqref{IIgoodcase} holds.

\noindent \textbf{Step 3:}
To discuss the remaining case, that is when the balls of $\F_{S_2}$ cover a great portion of the measure of $\O$, we need again to distinguish two different type of balls. More precisely, we introduce another constant $M>m$ whose precise value we explicit later (notice that the value of $m$ is not be fixed yet) and we denote by $\F_{S_2}'$ the family of the balls $\F_{S_2}$ for which 
\be\label{defultimepalle}
m<\max_{B(y_i,R_0)}w<M
\ee
holds true and we first discuss the case when
\be\label{IIIgoodcase}
\sum_{B\in \F'_{S_2}}|B_i|\ge \frac{1}{8}\frac{|\O|}{5^N}.
\ee
For any ball $B(x_i,r_i)\in \F'_{S_2}$ let $y_i^\star\in B(x_i,r_i)$ be such that $w(y_i^\star)\ge m$. Since, $r_i\le R_0$ we have $B(x_i,r_i)\subset B(y_i^\star,2R_0)$, and hence, again by the Vitali's covering Lemma, there exists a sequence $(y_i^\star)_{i\in I^\star}$ such that the disjoint balls 
$\{B(y_i^\star,2R_0)\}_{i\in I^\star}$ satisfy
\be\label{aiuto5}
\bigcup_{B\in \F'_{S_2}} B\sub \bigcup_{i\in I^\star} 5B(y_i^\star,2R_0).
\ee
By the inequality \eqref{caccioppoli2}, we have that, for any $r>0$ and $i\in I^\star$, it holds:
\be\label{aiuto4}
m\le w(y_i^{\star})\le C_1\left(\avint_{B(y_i^\star,r)}w^p(x)dx\right)^{1/p}+C_2r^{\frac{p}{p-1}}.
\ee
Here the constants $C_1$ and $C_2$ are completely determined and they depend only by $N$ and $p$. We choose $r_1>0$ to be such that
$$
C_2r_1^{\frac{p}{p-1}}\le \frac{m}{2},\quad r_1<r_0.
$$
Then, \eqref{defultimepalle} and \eqref{aiuto4} give
$$
\omega_N r_1^N\left(\frac{m}{2C_1}\right)^p\le \int_{B(y_i^\star,r_1)}w^p(x)dx\le M^p \left|\left\{w\ge \frac{3}{2}\right\}\cap B(y_i^\star,r_1)\right|+\left(\frac{3}{2}\right)^p \omega_N r_1^N.
$$
Now we fix the value of the constant $m$ in such a way that
$$
\left(\frac{m}{2C_1}\right)^p>2\left(\frac{3}{2}\right)^p .
$$
With this choice we have
\be\label{estimate1}
\frac{\left(\frac{3}{2}\right)^p \omega_Nr_1^N}{M^p} \le \left|\left\{w\ge \frac{3}{2}\right\}\cap B(y_i^\star,r_1)\right|.
\ee

In particular, by \eqref{aiuto5}, \eqref{IIIgoodcase} and \eqref{estimate1}, we have
\[
\begin{split}
&\frac 1 8 \frac{|\O|}{5^N} \le \sum_{B\in \F'_{S_{2}}} |B|\le 5^N\sum_{i\in I^\star} |B(y_i^\star,2R_0)|\\
&\le \frac{5^N M^p}{\left(\frac{3}{2}\right)^p \omega_Nr_1^N}\sum_{i\in I^\star} \left|\left\{w\ge \frac{3}{2}\right\}\cap B(y_i^\star,r_1)\right||B(y_i^\star,2R_0)|\\
&\le \frac{10^N M^p R_0^N}{\left(\frac{3}{2}\right)^p r_1^N}\sum_{i\in I^\star} \left|\left\{w\ge \frac{3}{2}\right\}\cap B(y_i^\star,r_1)\right|\le \frac{10^N M^p R_0^N}{\left(\frac{3}{2}\right)^p r_1^N}\left|\left\{w\ge \frac{3}{2}\right\}\right|,
\end{split}
\]
which implies \eqref{goal}. Thus the lemma is proved also if \eqref{IIIgoodcase} hold.

\noindent\textbf{Step 4:} We conclude the proof by discussing the only remaining case and showing that this leads to a contradiction for a suitable choice of the constant $M$. More precisely, suppose that neither \eqref{Igoodcase} nor \eqref{IIgoodcase} nor \eqref{IIIgoodcase} hold. Then, due to the construction we made of $\F$, it must certainly be 
\be\label{IVgoodcase}
\sum_{B\in \F_{S_2}\setminus\F'_{S_2}}|B_i|\ge \frac{1}{8}\frac{|\O|}{5^N}.
\ee
For every $B_i\in\F_{S_2}\setminus\F_{S_2}'$, we have
$$
M\le \max_{B(y_i,R_0)}w.
$$
Then, Lemma \ref{morrey} gives
$$
\frac{M^p}{R_0^{p-N}C^p}\le \int_{B(y_i,R_0)}|\nabla w(x)|^pdx
$$
for some constant $C=C(N,p)$.
As in the previous step we select a sequence $(y_i)_{i\in I^\star}$, such that the disjoint balls 
$(B(y_i,2R_0))_{i\in I^\star}$ satisfy
$$
\bigcup_{B\in \F\setminus\F'_{S_2}} B\sub \bigcup_{i\in I^\star} 5B(y_i,2R_0),
$$
to obtain
\[
\begin{split}
&\sum_{B\in \F\setminus\F'_{S_2}} |B|\le \sum_{i\in I^\star}|5B(y_i,2R_0)|\\
&\le\omega_N (5R_0)^N\sum_{i\in I^\star}\frac{\int_{B(y_i,2R_0)}|\nabla w(x)|^pdx}{\frac{M^p}{R_0^{p-N}C^p}}\\
&\le \frac{\omega_N (5R_0)^N}{\frac{M^p}{R_0^{p-N}C^p}}\int_\O |\nabla w(x)|^pdx.
\end{split}
\]
Thanks to \eqref{IVgoodcase}, the latter leads to
\be\label{estimate2}
\frac{1}{8}\frac{|\O|}{5^N}\le \frac{\omega_N (5R_0)^N}{\frac{M^p}{R_0^{p-N}C^p}}\int_\O |\nabla w(x)|^pdx.
\ee
Now we recall that $w$ satisfies \eqref{torsionprop} and  the assumption $w_0=1$, thus
$$
\int_\O |\nabla w(x)|^pdx=\int_{\O}w(x)dx=|\O|,
$$
which combined with \eqref{estimate2} finally allows to achieve
$$
\frac{1}{8}\frac{|\O|}{5^N}\le\frac{\omega_N (5R_0)^N}{\frac{M^p}{R_0^{p-N}C^p}}|\O|.
$$
The latter is a contradiction as soon as
$$
M^p> \max\{\omega_N8(25)^NR_0^pC^p, m\},
$$
and this concludes the proof.
\end{proof}

We now consider following ratio for $N<p<+\infty$:
$$
\Psi_p(\O)=\left(\frac{1}{|\O|}\int_{\O} w_{p,\O}(x)\right)\left(\frac{1}{|\O|}\int_{\O}w^p_{p,\O}(x)\right)^{-1/p}.
$$
Clearly, one has 
\be\label{trivial}
\Psi_p(\O)\ge \Phi_p(\O) \hbox{ for every }\O\sub\R^N \hbox{ open set with }0<|\O|<+\infty.
\ee
We also recall the following ``quantitative'' version of Jensen inequality.
\begin{lemm}\label{quantitativeJensen}
Let $p\ge 2$  and $E\subset\R^N$ any measurable set with $0<|E|<\infty$. Then 
$$
\avint_E (f(x))^pdx \ge\left(\avint_E f(x)dx\right)^p+\frac{1}{2^{p-1}-1}\avint_{E}\left|f(x)-\avint_E f(y)dy\right|^pdx
$$
for any Borel positive function on $E$.
\end{lemm}
\begin{proof}
For any $a,b\in \R$ we have, see for instance \cite{Li} Lemma $4.2$,
\be\label{Li}
|a|^p\ge |b|^p+p|b|^{p-2}b(a-b)+\frac{|a-b|^p}{2^{p-1}-1}.
\ee
By applying \eqref{Li} with $$a=f(x),\quad b=\avint_E f(x)dx,$$ and by integrating with respect to $x\in E$ we easily achieve the thesis. 
\end{proof}

We are now in a position to prove the following result.

\begin{theo}\label{maintorsion}
Let $N<p<+\infty$. Then we have 
$$\sup\left\{\Psi_p(\O): \ \O\subset\R^N,\hbox{ open set with }  0<|\O|<+\infty\right\}<1.$$
\end{theo}
\begin{proof}
With the same notation used in the proof of Lemma \ref{lemmakey}, i.e.
$$
w_{0}=\frac{1}{|\O|}\int_{\O}w(x)dx,\quad E_1=\{x\ :\ w(x)<w_0/2\}, \quad E_2=\{x\ : \ w(x)>3w_0/2\}.
$$
and by  possibly rescaling to get $w_0=1$ we can assume, by Lemma \ref{lemmakey} that
$$
|E_1|\ge \frac{|\O|}{C},
$$
for some constant $C>0$ that does not depend on $\O$ (the other case being similar). Then, by Lemma \ref{quantitativeJensen}, we get
\[
\begin{split}
&\avint_\O w^p(x)dx \ge w_0^p+\frac{1}{2^{p-1}-1}\avint_{\O}\left|w(x)-w_0\right|^pdx\\
& \ge w_0^p+\frac{1}{|\O|(2^{p-1}-1)}\int_{E_1}\left|w(x)-w_0\right|^pdx\ge w_0^p\left(1+\frac{|E_1|}{|\O|2^p(2^{p-1}-1)}\right)\\
& \ge w_0^p\left(1+\frac{1}{C2^p(2^{p-1}-1)}\right).
\end{split}
\]
By the very definition of the ratio $\Psi_p(\O)$, this suffices to prove the thesis. 
\end{proof}

As an immediate consequence of Theorem \ref{maintorsion} we have the following result concerning the efficiency of the torsion function when $p>N$.
\begin{coro}\label{coro.maintorsion}
Let $ N<p<+\infty$, then we have
$$\sup\left\{\Phi_p(\O): \ \O\subset\R^N,\hbox{ open set with }  0<|\O|<+\infty\right\}<1.$$
\end{coro}
\begin{proof}
Is enough to apply Theorem \ref{maintorsion} together with \eqref{trivial}.
\end{proof}

We conclude this section by proving the second part of Theorem \ref{brbu04},  corresponding to the particular case $p=+\infty$. The proof basically follows the same lines of that of Theorem \ref{maintorsion} (and Lemma \ref{lemmakey}), but with less technicalities, hence we only sketch it.

\begin{prop}\label{maindistance}
We have $$\sup\left\{\Phi_\infty(\O): \ \O\subset\R^N,\hbox{ open set with }  0<|\O|<+\infty\right\}<1.$$
\end{prop}
\begin{proof}
We use the following notation:
$$
d_0=\frac{1}{|\O|}\int_{\O}d(x,\O^c)dx,\quad E=\left\{x\in \O\ :\ 0\le d(x,\O^c)\le \frac{d_0}{2}\right\}.
$$
Our goal is to prove that there exists a constant $C=C(N)$ it holds:
\be\label{goaldist}
|E|\ge \frac{|\O|}{C}.
\ee
%Indeed if \eqref{goaldist} holds, then by applying Lemma \ref{quantitativeJensen} with $p=2$ we get
%$$
%\left(\avint_\O d^2(x,\O^c)\right)^{1/2}\ge d_0\sqrt{1+\frac{|E|}{4|\O|}}.
%$$
%which implies
%$$
%\|d(\cdot,\O)\|_{L^{\infty}(\O)}\ge \left(\frac{1}{|\O|}\int d(x,\O)dx\right)\sqrt{1+\frac{1}{4C}},
%$$
%and hence
%$$
%\Phi_{\infty}(\O)\le \left(1+\frac{1}{4C}\right)^{-\frac{1}{2}}
%$$
%for every $\O$, open set with finite measure.

Let $\F$ be the same family of of balls defined in \eqref{vitalicovering}. Notice that, for any $B_i\in\F$, we have
\be\label{estimate}
\int_{B(x_i,r_i)}d(x,\O^c)dx\ge \int_{B(x_i,r_i)}(r_i-|x-x_i|)dx=\frac{\omega_Nr_i^{N+1}}{N+1}.
\ee

Let denote by $\F_{L}$ the family of large balls defined through 
$$\F_L=\{B(x_j,r_j)\ : r_j>Kd_0\},$$ 
where $K>0$ is a large constant.
Then, we have:
\be\label{case1}
\sum_{B_j\in \F_L}|B(x_j,r_j)|\le \frac{|\O|}{5^N2},
\ee
since, if this is not the case, \eqref{estimate} would implies
\[
\begin{split}
&\sum_{B_j\in \F_L}\left(\int_{B(x_j,r_j)}d(x,\O^c)dx\right)\ge \sum_{j\in J}\frac{\omega_Nr_j^{N+1}}{N+1} \ge \frac{Kd_0}{N+1}\sum_{j\in J}\omega_N r_j^{N}\\
&=\frac{Kd_0}{(N+1)}\sum_{j\in J}|B(x_j,r_j)|,
\end{split}
\]
and hence
$$
d_0|\O|> \frac{Kd_0}{5^N2(N+1)}|\O|,
$$
which is a contradiction if we choose $K>5^N2(N+1)$.
By \eqref{case1}, we get that
\be\label{conseguenza}
\sum_{B_i\in \F\setminus \F_L} |B(x_i,r_i)|\ge \frac{|\O|}{5^N2}.
\ee
Now, for any small ball $B_i=B(x_i,r_i)\in \F\setminus \F_L$ there are two possibilities: either $r_i<\frac{d_0}{4}$ or $\frac{d_0}{4}<r_i<Kd_0.$
Clearly, when $r_i\le \frac{d_0}{4}$ it holds $B_i\subset E$,
while, if $r_i\ge d_0/4$, then there exists a region of volume at least $\omega_N(\frac{d_0}{4})^N$ of the ball $B_i$ contained in $E$.
Therefore
\[
\begin{split}
&|E|\ge \sum_{B_i\in \F\setminus \F_L, r_i\le d_0/4}\omega_N r_i^N+\sum_{B_i\in \F\setminus \F_L, r_i\ge d_0/4}\omega_N \left(\frac{d_0}{4}\right)^N\\
&\ge \sum_{B_i\in \F\setminus \F_L, r_i\le d_0/4}\omega_N r_i^N+\sum_{B_i\in \F\setminus \F_L, r_i\ge d_0/4}\omega_N\left(\frac{r_i}{4K}\right)^N
\end{split}.
\]
The latter, combined with \eqref{conseguenza} easily implies \eqref{goaldist}, and allows to conclude.
\end{proof}
\medskip
\noindent{\bf Proof of Theorem \ref{brbu04}.} This is a consequence of Corollary \ref{coro.maintorsion} and Proposition \ref{maindistance}. \qed

%%%%%%%%%%% Qui comincia il caso p=infty%%%%%%%%%%%%%%%%%%%%%%%%%%%%

\section{Proof of Theorem \ref{brbu05}: the honeycomb structure}\label{sdist}

We devote this  Section to prove Theorem \ref{brbu05} and determine the sharp value of $\sup \Phi_{\infty}$ in the $2$ dimensional case. 
The first tool we need in order to present the proof of Theorem \ref{brbu05} is a sort of isoperimetric-type property for  some geometric energy on triangles.  Given a triangle $\Delta\subset\R^2$ (we identify $\Delta$ with its interior), we denote by $r(\Delta)$ the radius of the  circle circumscribed to $\Delta$, by $V(\Delta)$ the set of its vertices and by ${\mathcal E}(\Delta)$ the following quantity:
$$
{\mathcal E}(\Delta)=\frac{\int_{\Delta}d(x,V(\Delta))dx}{|\Delta|r(\Delta)}.
$$
Notice that ${\mathcal E}(\cdot)$ is scaling free, in the sense that ${\mathcal E}(t\Delta)={\mathcal E}(\Delta)$, for any $t>0$ and any triangle $\Delta\subset\R^2$. Our first goal, which is of technical nature, is to show that among all the triangles the equilateral ones maximize ${\mathcal E}(\Delta)$.
We need the following lemma.

\begin{lemm} \label{reductionacute}
Let $\Delta\subset\R^2$ be any triangle. Then, there exists a triangle $\Delta'\subset\R^2$, which is either  acute\footnote{We do not consider a right triangle acute.} or right and isosceles, for which ${\mathcal E}(\Delta)\le {\mathcal E}(\Delta')$. 
\end{lemm}

\begin{proof}
%First we notice that, being $\|d(x,V(\Delta))\|_{L^{\infty}(\Delta)}\le 2r(\Delta)$ for every $\Delta\subset\R^2$, so that
%\be\label{supfinito}
%\sup\{E(\Delta)\ :\ \Delta\subset\R^{2},\ \Delta \hbox{ triangle}\}<2.
%\ee
%We denote by $E^{\star\star}$ the supremum defined in the left hand side of \eqref{supfinito}. 
Let $\Delta_{ABC}$ be a triangle with vertices $V(\Delta_{ABC})=\{A,B,C\}$, we use the notation 
$$
r_{ABC}=r(\Delta_{ABC}),\quad  {\mathcal E}_{ABC}={\mathcal E}(\Delta_{ABC}).
$$
Also we denote respectively by $\alpha,\beta$ and $\gamma$ the angles at the vertices $A,B$ and $C$. To prove the Lemma, we suppose that $\alpha\ge  \pi/2$. We reflect $\Delta_{ABC}$ with respect to the side $\overline{BC}$ and we denote by $A'$ the reflection of the vertex $A$. By the very construction we have
$$
2\int_{\Delta_{ABC}}d(x, \{A,B,C\})dx\le \int_{\Delta_{ABA'}}d(x, \{A,B,A'\})dx+\int_{\Delta_{ACA'}}d(x, \{A,C,A'\})dx.
$$
In particular
$${\mathcal E}_{ABC}\le \frac{1}{r_{ABC}}\left(\frac{\int_{\Delta_{ABA'}}d(x, \{A,B,A'\})dx+\int_{\Delta_{ACA'}}d(x, \{A,C,A'\})dx}{|\Delta_{ABA'}|+ |\Delta_{ACA'}|}\right)
$$
that is
\be\label{refl}
{\mathcal E}_{ABC}\le \left(\frac{\max\{r_{ABA'},r_{ACA'}\}}{r_{ABC}}\right)\max\{{\mathcal E}_{ABA'},{\mathcal E}_{ACA'}\}.
\ee
Notice that
$$
\frac{r_{ABA'}}{r_{ABC}}=\frac{|AB|\sin(\beta)}{2\sin(\beta)\cos(\beta)r_{ABC}}=\frac{|AB|}{cos(\beta)2r_{ABC}}=\frac{\sin(\gamma)}{\cos(\beta)}=\frac{\sin(\gamma)}{\sin(\alpha-\pi/2+\gamma)}\le 1,$$
the last inequality justified by the fact that $\gamma\le \alpha-\pi/2+\gamma<\pi/2$.
Therefore
$$
\frac{r_{ABA'}}{r_{ABC}}=\frac{\sin(\gamma)}{\sin(\alpha-\pi/2+\gamma)}\le 1
$$
and analogously
$$
\frac{r_{ACA'}}{r_{ABC}}= \frac{\sin(\beta)}{\sin(\alpha-\pi/2+\beta)}\le 1.
$$
Combing these inequalities with \eqref{refl} we get
$$
{\mathcal E}_{ABC}\le \max\{{\mathcal E}_{ABA'},{\mathcal E}_{ACA'}\}.
$$
Suppose without loss of generality that $\Delta_{ACA'}$ reaches the maximum in the inequality above.
Being $\Delta_{ACA'}$ isosceles, if $2\gamma< \pi/2$, then is also acute and the lemma is proved.
Suppose instead that $2\gamma> \pi/2$, then we can repeat the same argument as above: we reflect the triangles $\Delta_{ACA'}$ with respect to the side $\overline{AA'}$, we denote by $C'$ the reflection of the vertex $C$. Again we obtain
$$
{\mathcal E}_{ACA'}\le \max\{{\mathcal E}_{CC'A'},{\mathcal E}_{C'CA}\}.
$$
Both the triangles $\Delta_{CC'A'}$ and $\Delta_{C'CA}$ are again isosceles, and in this case they both have two angles equal to $\gamma$. Hence being $2\gamma>\pi/2$ they are  acute triangles. The only remaining case is when $2\gamma=\pi/2$ which precisely corresponds to the case when $\Delta_{ACA'}$ is a right isosceles triangle .
\end{proof}

We can prove now the isopertimetric-type inequality.

\begin{prop}\label{lemma2}
For every triangle $\Delta\subset\R^2$ we have 
\be\label{isoper}
{\mathcal E}(\Delta)\le {\mathcal E}(\Delta_{eq}),
\ee
where $\Delta_{eq}\subset\R^2$ is any equilateral triangle.
\end{prop}

\begin{proof}
By scaling and translation invariance and thanks to Lemma \ref{reductionacute}, it is enough to prove \eqref{isoper} among triangles inscribed in $B_1$ and for which the origin is contained in the closure of $\Delta$. For such a triangle we can easily  determine the value of ${\mathcal E}(\Delta)$. More precisely, by denoting with $\ell_{1},\ell_{2},\ell_{3}$ the lengths of the three sides of $\Delta$, we have
$$
\int_{\Delta}d(x,V(\Delta))dx=\frac{1}{12}\sum_{i=1}^{3}\ell^3_{i}\int_{0}^{\cos^{-1}(\ell_{i}/2)}\frac{d\theta}{\cos^3(\theta)},
$$
and since
$$
\int_0^{\tau}\frac{d\theta}{\cos^3(\theta)}=\int_{0}^{\sin(\tau)}\frac{dt}{(1-t^2)^2}=\frac{1}{4}\left(\frac{2\sin(\tau)}{\cos^2(\tau)}+\ln\left(\frac{1+\sin(\tau)}{1-\sin(\tau)}\right)\right),
$$
we have also
$$
\int_{0}^{\cos^{-1}(\ell/2)}\frac{d\theta}{\cos^3(\theta)}=\left(\frac{\sqrt{4-\ell^2}}{\ell^2}+\frac 1 4 \ln\left(\frac{2+\sqrt{4-\ell^2}}{2-\sqrt{4-\ell^2}}\right)\right)
$$
Being clearly
$$
\quad |\Delta|=\sum_{i=1}^3  \frac{ \ell_i \sqrt{4-\ell_i}}{4},
$$
we deduce
\be\label{acutevalue}
{\mathcal E}(\Delta)=\frac 1 3 +\frac 1 {12} \frac{\sum_{i=1}^3 \ell_i^3\ln \left(\frac{2+\sqrt{4-\ell_i^2}}{2-\sqrt{4-\ell_i^2}}\right)}{\sum_{i=1}^3 	\ell_i\sqrt{4-\ell_i^2}},
\ee
For the equilateral triangle inscribed in $B_1$ we have $\ell_1=\ell_2=\ell_3=\sqrt{3}$, and hence, by \eqref{acutevalue}:
\be\label{valueatequil}
{\mathcal E}(\Delta_{eq})=\frac 1 3 +\frac {\ln(3)} 4.
\ee
Hence, to prove \eqref{isoper}, we need to show that
\be\label{daprovare}
\sum_{i=1}^3 \ell_i^3\ln\left(\frac{2+\sqrt{4-\ell_i^2}}{2-\sqrt{4-\ell_i^2}}\right)-3\ln(3)\sum_{i=1}^{3}\ell_i\sqrt{4-\ell_i^2}\le 0.
\ee
for any triplet $(\ell_1,\ell_2,\ell_3)$ which corresponds either to an acute triangle inscribed in $B_1$ or to a right and isosceles triangle inscribed in $B_1$.
For the sake of brevity we denote by $\mathcal{L}(\ell_1,\ell_2,\ell_3)$ the left-hand side of \eqref{daprovare}.

Let $(\ell_{1n},\ell_{2n},\ell_{3n})$ any maximizing sequence for $\mathcal{L}$ and $\Delta_n\subset B_1$ the corresponding triangles. By Blashcke Selection Theorem, up to sub sequence, the sequence $\overline{\Delta}_n$ converges to some closed set $\overline{\Delta}_\infty$.  Also, since any $\overline{\Delta}_n$ contains the origin, the limit set $\overline{\Delta}_\infty$ contains the origin as well. 
In particular,  either $\overline{\Delta}_{\infty}$ is a diameter of $B_1$ or $\Delta_\infty$ is a right triangle or an acute triangle. 
Assume that $\overline{\Delta}_\infty$ is a diameter of $B_1$. Then we must have (possibly relabeling the indexes)
$$
\lim_{n\to\infty}\ell_{1n}=0, \quad \lim_{n\to\infty}\ell_{2n}=2,\ \lim_{n\to\infty}\ell_{3n}=2,
$$
and therefore 
$$
\lim_{n\to+\infty} \mathcal{L}(\ell_{1n},\ell_{2n},\ell_{3n})=0,
$$
proving, by the maximality of the sequence $(\ell_{1n},\ell_{2n},\ell_{3n})$, \eqref{daprovare} and consequently the proposition. 
Hence, let us suppose that $\Delta_\infty$ is a triangle.
If is a right triangle, then by Lemma \ref{reductionacute} must be also isosceles.
In this case we have, up to relabel indexes, $\ell_1=2$ and $\ell_2=\ell_3=\sqrt{2}$ and we readily verify that $\mathcal{L}(2,\sqrt{2},\sqrt{2})<0$.
Thus suppose that $\Delta_\infty$ is an acute triangle. Being so we can exploit optimality conditions and is convenient to introduce the angular coordinates $x_1,x_2,x_3>0,$  defined through $\cos(x_i)=\ell_i/2$, for $i=1,2,3$.
Clearly, any acute triangles inscribed in $B_1$ satisfies
\be\label{vincoli}
x_1,x_2,x_3\in (0,\pi/2), \quad x_1+x_2+x_3=\pi/2,
\ee
and we have
\be\label{angularcoordinate}
\mathcal{L}(\ell_1,\ell_2,\ell_3)=8\left(\sum_{i=1}^3 \cos^3(x_i)\ln\left(\frac{1+\sin(x_i)}{1-\sin(x_i)}\right)-\frac{3\ln(3)}{2}\sum_{i=1}^3\cos(x_i)\sin(x_i)\right).
\ee
%and applying \eqref{vincoli}, inequality \eqref{daprovare} is equivalent to
%\be\label{daprovare2}
%\tilde{\mathcal{L}}(x_1,x_2,\pi/2-(x_1+x_2))\le 0,\quad  (x_1,x_2)\in T,
%\ee
%where $T=\{(x_1,x_2)\in (0,\pi/2)^2\ :\ x_1+x_2<\pi/2\}\subset\R^2$ and we denote by $\tilde{\mathcal{L}}(x_1,x_2,x_3)$  the right hand side of \eqref{angularcoordinate}.

For the sake of brevity we denote by $f(x),g(x)$ the following quantities
$$
f(x)=\cos^3(x)\ln\left(\frac{1+\sin(x)}{1-\sin(x)}\right),\quad g(x)=\cos x \sin x.
$$
We have
$$
f'(t)=-3\cos^2(t)\sin(t)\ln\left(\frac{1+\sin t}{1-\sin t}\right)+2\cos^2 t,\quad 
g'(t)=\cos(2t).
$$
Differentiate the right hand side of \eqref{angularcoordinate} at a critical point, taking also into account of the constraints \eqref{vincoli}, gives the following conditions 
\be\label{condizioniottimalita}
f'(x_1)-\frac{3\ln(3)}{2}g'(x_1)=f'(x_2)-\frac{3\ln(3)}{2}g'(x_2)
=f'(x_3)-\frac{3\ln(3)}{2}g'(x_3)=\Lambda,
\ee
where $\Lambda>0$ is the Lagrange multiplier.
By studying the function 
$$
t\mapsto f'(t)-\frac{3\ln(3)}{2}g'(t)=\cos^2 t-\frac{3\ln(3)}{2}\cos(2t)-3\cos^2t\sin t\ln\left(\frac{1+\sin t}{1-\sin t}\right).
$$ 
we can show that \eqref{condizioniottimalita} implies that at least two of the three coordinates $(x_1,x_2,x_3)$ coincide. The proof of this claim is elementary 
and we give the details in Appendix \ref{sapp} Lemma \ref{isoscele}. Hence the triangle $\Delta_\infty$ has to be isosceles.

To conclude we show that that among isosceles triangles inscribed in $B_1$, $\mathcal{L}$ assume its maximum at the equilateral one.
Indeed let $\Delta$ be isosceles and let $\ell_1,\ell_2,\ell_3$ be its side lengths. Suppose that $\ell_1=\ell_2=\ell\in [\sqrt{2},2)$ (the interval is determined by the fact that we are considering acute triangles inscribed in $B_1$).
By Heron's formula we have
\be\label{Heron}
|\Delta|=\frac{\ell_3\sqrt{\ell^2-(\ell_3/2)^2}}{2}.
\ee
Since $|\Delta|=\ell_1\ell_2\ell_3/4$, \eqref{Heron} implies $
\ell_3=\ell\sqrt{4-\ell^2}$ and we can express ${\mathcal E}(\Delta)$ as  function of $\ell$ (writing ${\mathcal E}(\Delta)={\mathcal E}(\ell)$). Precisely, by \eqref{acutevalue}, we get
$$
{\mathcal E}(\ell)=\frac 1 3 \left(1-\frac{(4-\ell^2)}{4}\ln\left(\frac{4-\ell^2}{\ell^2}\right)+\frac{\ln\left(\frac{2+\sqrt{4-\ell^2}}{2-\sqrt{4-\ell^2}}\right)}{2\sqrt{4-\ell^2}}\right).
$$
Is convenient the change of variables  $\xi=\sqrt{4-\ell^2}\in (0,\sqrt{2})$. With such a choice the right-hand side of the latter identity is equal to the following function: 
$$
\xi\mapsto\frac{1}{3}\left(1-\frac{\xi^2}{4}\ln\left(\frac{\xi^2}{4-\xi^2}\right)+\frac{\ln\left(\frac{2+\xi}{2-\xi}\right)}{2\xi}\right),
$$
By optimizing the above expression with respect to $\xi\in (0,\sqrt{2})$ we obtain that the maximum value correspond to $\xi=1$, that is when  $\ell=\ell_3=\sqrt{3}$.
Again the details are elementary and we prefer to refer the reader to Lemma \ref{optimiso} in Appendix \ref{sapp}.
 Hence the optimal isosceles triangle has to be equilateral and this proves the thesis.
\end{proof}

An important  tool that we shall use the  proof of Theorem \ref{brbu05}  is the  Delaunay triangulation of a family of points \cite{De}. For reader convenience we briefly recall the main definitions and properties, following the monograph \cite{AKL} to which we refer for more details.

Given a  family of points $\mathcal{S}$, the Voronoi cell of a point $p\in\mathcal{S}$  is defined by
\be\label{voronoicell}
V(p)=\{x\in\R^2\ :\ |x-p|\le |x-q|, \hbox{ for every }q\in\mathcal{S}\}.
\ee
To construct $V(p)$ it is enough to consider for any other point $q\in\mathcal{S}$, the bisector line of the segment $\overline{pq}$ (i.e. the set of all points having the same distance from $p$ and $q$). Such a line  divides $\R^2$ into two half plane. By denoting with $H_{pq}$ the one containing  $p$, we have
$$
V(p)=\bigcap_{q\in\mathcal{S}}H_{pq}.
$$
In particular $V(p)$ is a convex, possibly unbounded polygon. The edges and the vertexes, of the family of polygons $V(p)$ are called Voronoi edges and Voronoi vertexes. The union of Voronoi edges and vertexes generate a planar straight-line graph, which is commonly known as the Voronoi diagram associated to $\mathcal{S}$. A Voronoi diagram naturally determine a partition of $\R^2$, made up of convex regions with mutually disjoint interior.

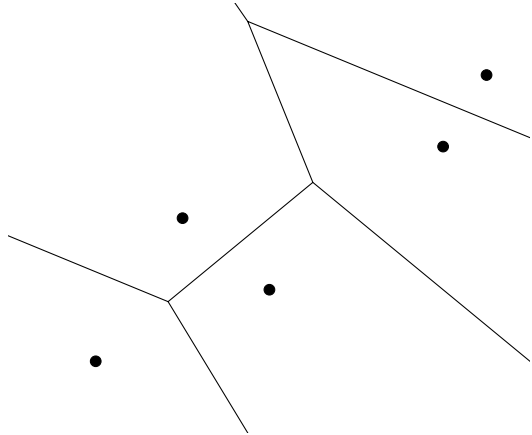
\begin{figure}[h]
\begin{tikzpicture}
\pgfplotsset{ticks=none}
\begin{axis}[mark=none,
xmin=0,
xmax=6,
ymin=0,
ymax=6,
hide axis, %uncomment to bound with square
]
\addplot [only marks, black] table {points.dat};
\addplot [no markers, update limits=false] table {voronoi.dat};
\end{axis}
\end{tikzpicture}
\caption{\textit{An example of Voronoi diagram spanned by five points}}
\end{figure}

Starting from the Voronoi diagram of  $\mathcal{S}$ one may build the associated Delaunay tessellation: this is the straight-line graph with vertex set $\mathcal{S}$ determined by saying that a segment connecting two points of $\mathcal{S}$ belongs to the graph if and only if the Voronoi regions $V(p)$ and $V(q)$ are edge-adjacent.
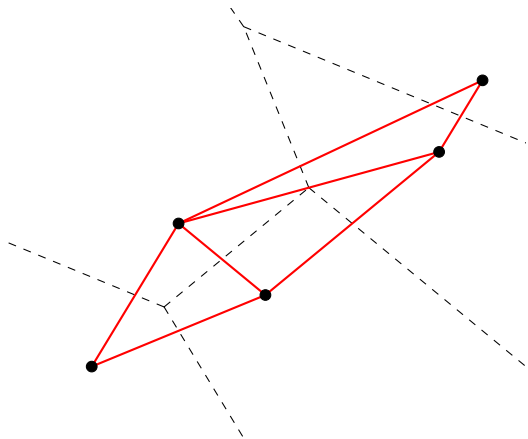
\begin{figure}[h]
\begin{tikzpicture}
\pgfplotsset{ticks=none}
\begin{axis}[mark=none,
xmin=0,
xmax=6,
ymin=0,
ymax=6,
hide axis, %%uncomment to bound with square
]
\addplot [only marks, black] table {points.dat};
\addplot [no markers, dashed, update limits=false] table {voronoi.dat};
\addplot [thick, red, domain= 1:2] {(x-1)*(3-1)/(2-1)+1};
\addplot [thick, red, domain= 1:3] {(x-1)*(2-1)/(3-1)+1};
\addplot [thick, red, domain=2:3]{(x-2)*(2-3)/(3-2)+3};
\addplot [thick, red, domain=2:5.5]{(x-2)*(5-3)/(5.5-2)+3};
\addplot [thick, red, domain=2:5]{(x-2)*(4-3)/(5-2)+3};
\addplot [thick, red, domain=3:5]{(x-3)*(4-2)/(5-3)+2};
\addplot[thick, red, domain=5:5.5]{(x-5)*(5-4)/(5.5-5)+4};
\end{axis}
\end{tikzpicture}
\caption{\textit{Delauney triangulation of five points, the dashed lines represent the Voronoi diagram.}}
\end{figure}

Any edge of the convex hull of $\mathcal{S}$ belongs to the Delaunay tessellation, although in some cases these are the only ones: consider for instance the case when all the points of $\mathcal{S}$ are co-linear or co-circular. 
If the family $\mathcal{S}$ does not lie in a single line, the Delaunay tessellation define a partition of the convex hull of $\mathcal{S}$ made up of convex polygons (called faces) which satisfy the \textit{empty-circle property}: the circle that circumscribes any polygon does not contain, in its interior, any other point of $\mathcal{S}$.

In general the  faces determined by the Delaunay tessellation can be polygons other than triangles (consider again the case of four co-circular points), however we can always add to the graph new edges to obtain a new graph which has only triangular faces, and for which the empty-circle property holds as well. 
We call any  graph obtained in such a way a Delaunay triangulation of $\mathcal{S}$.
The notions of Delaunay edge, vertex and graph come naturally with this last definition.

In the sequel, given $\eps>0$, we denote by $\eps\mathbb{Z}^2$ the usual lattice of points, given by
$$\eps\mathbb{Z}^2=\{(\eps i,\eps j)\subset\R^2:\ (i,j)\in \mathbb{Z}^2\}.$$

\begin{lemm}\label{delauney1}
Let $\eps>0$ and $\mathcal{S}\subset \eps\mathbb{Z}^2$ a finite set.  If $p,q\in \mathcal{S}$, and $|p-q|=\eps$, then any Delauney triangulation of $\mathcal{S}$ contains the segment connecting $p$ and $q$.
\end{lemm}

\begin{proof}
Is enough to notice that, being $\mathcal{S}\subset\eps\mathbb{Z}^2$, the middle point  $m\in \R^2$ of the segment connecting $p$ and $q$, belongs to both $V(p)$ and $V(q)$ (which are the Voronoi cells defined through \eqref{voronoicell}), and cannot belong to some other Voronoi cells, to deduce that $V(p)$ and $V(q)$  are edge-adjacent. 
\end{proof}

We denote by $\mathcal{Q}_\eps$ the family of all the finite union of closed squares, of size $\eps$ and vertices in $\eps\mathbb{Z}^2$, that is
$$
\mathcal{Q}_\eps=\left\{\bigcup_{(i,j)\in I\times J} \eps([i-1,i]\times[j-1,j]) :\ I\times J\subset \mathbb{Z}^2,\ I\times J\ \hbox{finite sets of indexes}	\right\}.
$$
Given $Q\in\mathcal{Q}_\eps$ we define its $\eps$\textit{-discrete boundary} as $\pa_{d,\eps}Q=\pa Q\cap \eps\mathbb{Z}^2.
$ 

\begin{lemm}\label{delauney2}
Let $\eps>0$ and $Q\in\mathcal{Q}_\eps$. Then, any Delauney triangle of any Delauney triangulation of $\pa_{d,\eps}Q$ has interior which lies either in $Q$ or in the interior of $Q^c$. 
\end{lemm}

\begin{proof}
The thesis follows from the fact that, by Lemma \eqref{delauney1} the whole boundary $\pa Q$ is made of Delauney edges, and hence cannot be crossed by any triangle.
\end{proof}

\begin{lemm}\label{delauney3}
Let $\eps>0$ and $Q\in\mathcal{Q}_\eps$. Let $\Delta$ be a Delauney triangle $\Delta$ of a Delauney triangulation  of $\pa_{d,\eps}Q$. If $\Delta\subset Q$, then the center of the circle circumscribed to $\Delta$ belongs to $Q$ as well.
\end{lemm}

\begin{proof}

Let us denote by $x_0$ be the center of the circle circumscribed to $\Delta\subset Q$ and suppose by contradiction that $x_0\not\in Q$. Then, $\Delta$ is an obtuse triangle. Let $A$ and $B$ be the vertices of the side faced to the largest angle. Let also $A'$ and $B'$ be the projection of $A$ and $B$ on the diameter parallel of the circle circumscribed to $\Delta$, parallel to the side of $\Delta$ connecting $A$ and $B$.

The side connecting $A$ and $B$ has not minimal length, and so, by Lemma \ref{delauney1}, cannot be a boundary side. Therefore, if we consider the region $R\subset\R^2$, determined by $A,B,A'$ and $B'$, there exists $z\in \pa Q\cap R$ such that
$$
|z-x_0|=\min\{|x-x_0|:\ x\in Q\}.
$$ 

%\begin{tikzpicture}
%\draw [fill=none] (0,0) circle (3cm);
%\draw[fill=black] (0,0) circle (1 pt) node [above] {\tiny $x_0$};
%\draw[fill=black] (1.5,2.598) circle (1 pt) node [above] {\tiny A};
%\draw[fill=black] (-2.121,2.121) circle (1 pt) node [above] {\tiny B};
%\draw[fill=black] (0,3) circle (1 pt) node [above] {\tiny C};
%\draw (1.5,2.598)--(-2.121,2.121)--(0,3)--(1.5,2.598);
%\draw (2.973,0.389)--(-2.973,-0.389);
%\end{tikzpicture}

The point $z$ has to be contained in the interior of a boundary side of length $\eps$. However, by the empty circle property of the triangle $\Delta$ the vertices of such a segment cannot be contained in the circle circumscribed to $\Delta$. Hence the length of this segment has to be greater then the length of the segment connecting $A$ and $B$, in contradiction whit is minimality.

\end{proof}

We are now in a position to prove Theorem \ref{brbu05}.
\par
\noindent{\bf Proof of Theorem \ref{brbu05}.}
In order to estimate form above the value $\Phi_{\infty}(\O)$, by possibly replacing $\O$ with $\O\cap B_R$ and sending $R\to\infty$, without any loss of generality, we can clearly assume $\O$ to be a bounded open set.

Let $\eps>0$, we define $\overline{\O}_\eps$ to be the following closed set,
$$
\overline\O_\eps:=\bigcup_{Q\in \mathcal{Q}_\eps, Q\Subset \O}Q,
$$
and $\O_\eps$ to be the interior of $\overline{\O}_\eps$. 
Clearly we have
$\chi_{\O_\eps}(\cdot)\to \chi_{\O}(\cdot)$ pointwise (and in $L^1$).
Also, being $\bigcup_{\eps>0}\O_\eps=\O$  we also have 
$$\lim_{\eps\to 0}d_{H}(\O_\eps^c,\O^c)=0,$$ 
where $d_H$ denotes the usual Hausdorff metric (see for instance \cite{He}). The latter in particular implies that $d(\cdot,\O_\eps^c)$ converge to $d(\cdot,\O^c)$ in $L^{\infty}(\R^N)$, and consequently that
\be\label{polygapp}
\lim_{\eps\to 0}\Phi_{\infty}(\O_\eps)= \Phi_{\infty}(\O).
\ee
We define the following quantity:
$$
\Phi_{d,\infty}(\O_\eps)=\frac{\int_{\O_\eps}d(x,\pa_{d,\eps}\O_\eps)dx}{|\O_\eps|\|d(\cdot,\pa_{d,\eps}\O_\eps)\|_{L^{\infty}(\O_\eps)}}.
$$
Notice that, for every $x\in \O_\eps$, it hold:
\be\label{bndryapp}
d(x,\pa\O_\eps)\le d(x,\pa_{d,\eps}\O_\eps), \quad d^2(x,\pa_{d,\eps}\O_\eps)\le d^2(x,\pa\O_\eps)+\frac{\eps^2}{2}.
\ee

By \eqref{bndryapp}, we deduce that, if there exists $\eps_0>0$ and $m>0$ such that $\Phi_{d,\infty}(\O_\eps)\le m$ for every $\eps<\eps_0$, then $\Phi_{\infty}(\O)\le m$. Indeed, if such an $\eps_0$ exists then for every $\eps<\eps_0$ we have
\[
\begin{split}
&\Phi_{\infty}(\O_\eps)=\frac{\int_{\O_\eps}d(x,\pa\O_\eps)dx}{|\O_\eps| \|d(\cdot,\pa\O_\eps)\|_{L^{\infty}(\O_\eps)}}\le \frac{\int_{\O_\eps}d(x,\pa_{d,\eps}\O_\eps)dx}{|\O_\eps|\|d^2(\cdot,\pa_{d,\eps}\O_\eps)-\frac{\eps^2}{2}\|^{1/2}_{L^{\infty}(\O_\eps)}}\\
&\le \Phi_{d,\infty}(\O_\eps)\frac{\|d(\cdot,\pa_{d,\eps}\O_\eps)\|_{L^{\infty}(\O_\eps)}}{\|d^2(\cdot,\pa_{d,\eps}\O_\eps)-\frac{\eps^2}{2}\|^{1/2}_{L^{\infty}(\O_\eps)}}\le m\frac{\|d(\cdot,\pa_{d,\eps}\O_\eps)\|_{L^{\infty}(\O_\eps)}}{\|d^2(\cdot,\pa_{d,\eps}\O_\eps)-\frac{\eps^2}{2}\|^{1/2}_{L^{\infty}(\O_\eps)}}.
\end{split}
\]
Passing to the limit as $\eps\to 0$ and taking into account of \eqref{polygapp}, we obtain the desired inequality.
Hence we can focus on proving that, for $\eps$ small enough it holds
\be\label{disapp}
\Phi_{d,\infty}(\O_\eps)\le \frac 1 3 +\frac {\ln(4)} 3.
\ee

Now, we fix $\eps_0>0$ to be sufficiently small such that, if $\eps<\eps_0$, the family $\pa_{d,\eps}\O_\eps$ contains at least three non co-linear points, we fix a Delaunay triangulation of $\pa_{d,\eps}\O_\eps$, and we denote with $\mathcal{F}$ the family of the triangles contained in $\O_\eps$ of  the chose Delauney triangulation (which is well defined by virtue of Lemma \ref{delauney2}). 
We notice that for every $\Delta\in \mathcal{F}$, and $x\in \Delta$, the empty-circle property gives 
$d(x,\pa_{d,\eps}\O_\eps)=d(x,V(\Delta))$.
Furthermore by Lemma \ref{delauney3} we also get $r(\Delta)\le \|d(\cdot,\pa_{d,\eps}\O_\eps)\|_{L^{\infty}(\O_\eps)}$, where we recall that $r(\Delta)$ denotes the radius of the circle circumscribed to $\Delta$.

Therefore, we have
\[
\Phi_{d,\infty}(\O_\eps)=\frac{\sum_{\Delta\in\mathcal{F}} \int_{\Delta}d(x,\pa_{d,\eps}\O_\eps)dx}{\left(\sum_{\Delta\in\mathcal{F} }|\Delta|\right)\|d(\cdot,\pa_{d,\eps}\O_\eps)\|_{L^{\infty}(\O_\eps)}}\le \frac{\sum_{\Delta\in\mathcal{F}}{\mathcal E}(\Delta)|\Delta|}{\sum_{\Delta\in\mathcal{F} }|\Delta|}
\]
and in particular
$$\Phi_{d,\infty}(\O_\eps)\le \max_{\Delta\in \mathcal{F}}{\mathcal E}(\Delta).$$
The conclusion follows at once by applying Proposition \ref{lemma2}, to obtain
$$
\Phi_{d,\infty}(\O_\eps)\le {\mathcal E}(\Delta_{eq}),
$$
which recalling \eqref{valueatequil} implies \eqref{disapp} and, hence, 
$$
\Phi_{\infty}(\O)\le \frac{1}{3}+\frac{\ln(4)}{3}.
$$
To conclude the proof is enough to consider the following construction. Let $E\subset\R^2$ be the hexagon centered at the origin with unitary side. For every $\eps>0$ we fix $C_\eps\subset\R^2$ to be a set of points $y$ such that the family $\{\eps E+y\}_{y\in C_\eps}$ produces  an hexagonal tiling for $\R^2$. We then define $\O_\eps$  as 
$$
\O_\eps= B(0,1) \setminus \left\{y\in C_\eps: \eps E+y \Subset B(0,1) \right\}.
$$
As $\eps\to 0$, the effects of the boundary of $B(0,1)$ become negligible and one has
$$
\lim_{\eps\to 0}\Phi_{\infty}(\O_\eps)=\frac{1}{\eps}\left(\frac{|\O|}{|\eps E|}\right)\frac{\int_{\eps E}|x|dx}{|\O|}=\avint_{E}|x|dx=\frac{1}{3}+\frac{\ln(3)}{4}.
$$
This concludes the theorem.\qed

\section {Further remarks and applications}\label{brbu10}

As a byproduct of our analysis we obtain some information concerning  a shape optimization problem that recently has been investigated  : those of comparing the torsional rigidity $T_p(\O)$, defined as  
\be\label{torsionalrigidity}
T_p(\O)=\left(\int_\O w_{p,\O}(x)dx\right)^{p-1},
\ee
 and the principal frequency of the $p$-Laplace operator $\la_p(\O)$, variationally characterized by means of the following Rayleigh quotient:
\be\label{eigenvalue}
\la_p(\O)=\inf\left\{\frac{\int_\O|\nabla u(x)|^pdx}{\int_\O|u(x)|^pdx}: u\in W^{1,p}_0(\O)\setminus \{0\}\right\}.
\ee

A consequence  Theorem \ref{brbu04} is that that for $p>N$ 
$$
\sup\left\{\frac{\la_p(\O)T_p(\O)}{|\O|^{p-1}}:\ \O\sub\R^N,\ \hbox{open set with}\ 0<|\O|<+\infty\right\}<1.
$$
This supremum above was proved  to be equal to $1$  in \cite{befe} for $p=2$ and in  \cite{bribu} for $p\le N$.
The following Corollary gives a positive answer to the Open problem $2$ of \cite{bribu}.
\begin{coro}
Let $p>N$ and $F_p(\O)$ be the shape functional defined by 
$$
F_{p}(\O)=\frac{\la_p(\O)T_p(\O)}{|\O|^{p-1}},
$$
where $T_p(\O)$ and $\la_p(\O)$ are respectively defined through \eqref{torsionalrigidity} and \eqref{eigenvalue}.
Then 
$$\sup \left\{F_{p}(\O):\ \O\subset\R^N,\hbox{ open set with }  0<|\O|<+\infty\right\}<1.$$
\end{coro}
\begin{proof}
For every $\O\subset\R^N$ open set of positive and finite measure, we have
$$
\avint_{\O}|w_{p,\O}(x)|dx\le \sup_{\O}\Psi_p(\O)\left(\avint_{\O}|w_{p,\O}(x)|^pdx\right)^{1/p}
$$
Since $w_{p,\O}\in W^{1,p}_0(\O)$ is admissible as a test function for $\la_p(\O)$  we have
$$
\la_p(\O)\le \frac{\int _\O|\nabla w_{p,\O}(x)|^pdx}{\int_{\O}|w_{p,\O}(x)|^pdx}.$$
Taking also into account of \eqref{torsionprop} we obtain
$$\la_p(\O)\le \frac{\avint_\O | w_{p,\O}(x)|dx}{\avint_{\O}|w_{p,\O}(x)|^pdx}\le \frac{|\O|^{p-1}\left(\sup_\O\Psi_p(\O)\right)^p}{\left(\int_{\O}|w_{p,\O}(x)|dx\right)^{p-1}}.
$$
This precisely means
$$
F_p(\O)\le \left(\sup\left\{\Psi_p(\O): \ \O\subset\R^N,\hbox{ open set with }  0<|\O|<+\infty\right\} \right)^{p},
$$
and the thesis follows by Proposition \ref{maintorsion}.
\end{proof}
\begin{rem}
It is interesting to recall that when $\O$ is a bounded convex domain, the following inequalities hold true
\be\label{resulthelupiCONVEX}
\frac{1}{(N+1)^2}\le \Phi_2(\O)\le \frac{2}{3}.
\ee
The convex case is exhaustively extended for  $p\neq 2$ in \cite{DeGaGu}, where the authors show that
\be\label{resultDeGaGu}
\frac{p'}{N^{p'-1}(N+p')}\le \Phi_p(\O)\le \frac{p'}{p'+1}.
\ee

The right side inequality, in both \eqref{resulthelupiCONVEX} and \eqref{resultDeGaGu}, is sharp, and equality can be obtained by a suitable sequence of thinning rectangles type domains while the sharp lower bound to $\Phi_p$, among convex domains, is up to our knowledge still not known (notice that when $p=2$, the results of \cite{DeGaGu} slightly improves the lower bound given in \eqref{resulthelupiCONVEX}). 

Inequalities similar to those  in \eqref{resultDeGaGu} are obtained in \cite{DeGaGu} also in the more general setting of the anisotropic $p$-Laplace operators.

It is worth to notice that, in \cite{HeLuPi}, the value $(N+1)^{-1}$ is conjectured to be the sharp lower bound to $\Phi_2$ among planar convex sets.
 \end{rem}

\appendix\label{sapp}
\section{}

We give here the proof of the elementary computations we used to prove Proposition \ref{lemma2}.

\begin{lemm}\label{isoscele}
Let $\mathrm{L}:(0,\pi/2)\to \R$ be defined by 
$$
\mathrm{L}(t)=2 \cos^2 t-\frac{3\ln(3)}{2}\cos(2t)-3\cos^2t\sin t\ln\left(\frac{1+\sin t}{1-\sin t}\right).
$$
If $t_1,t_2,t_3$ are such that $\mathrm{L}(t_1)=\mathrm{L}(t_2)=\mathrm{L}(t_3)$, then either $t_1=t_2$ or $t_1=t_3$ or $t_2=t_3$.
\end{lemm}
\begin{proof}
One has 
$$
\mathrm{L}'(t)=(3\ln(3)-5)\sin(2t)-\ln\left(\frac{1+\sin(t)}{1-\sin(t)}\right)(9\cos^3(t)-6\cos(t)).
$$
In particular $\mathrm{L}'(t)\le 0$ if and only if
$$
(6\ln(3)-10)\sin(t)\ge \ln\left(\frac{1+\sin(t)}{1-\sin(t)}\right)(9\cos^2(t)-6)
$$
We denote respectively by $H_1(t)$ and $H_2(t)$ the left hand side and the right hand side of the previous inequality.
Notice that
$$
H_2(t)\ge 0 \Leftrightarrow 9\cos^2(t)-6\ge 0\Leftrightarrow t\le \cos^{-1}\left(\sqrt{\frac 2 3 }\right)
$$
and 
$$
H_2'(t)=\frac{2}{\cos(t) }\left(9\cos^2(t)  - 6 \right)-18 \ln\left(\frac{1+\sin (t)}{1-\sin (t)}\right)\cos (t) \sin (t) ,
$$
so that, when $t>\cos^{-1}(\sqrt{\frac{2}{3}})$, we get $H_2'(t)<0$. Being $H_1(t)\le 0$ this implies that there exists $t_0\in (0,\pi/2)$, such that 
$\mathrm{L}'(t)$ is negative for $t\le t_0$ and positive otherwise, which clearly imply the thesis.

\end{proof}

\begin{lemm}\label{optimiso}
Let  $f:\left[0,\sqrt{2}\right ]\to \R$  be defined by
$$
f(\xi)=\frac{1}{3}\left(1-\frac{\xi^2}{4}\ln\left(\frac{\xi^2}{4-\xi^2}\right)+\frac{\ln\left(\frac{2+\xi}{2-\xi}\right)}{2\xi}\right).
$$
Then $f$ reaches its maximum at $\xi=1$.
\end{lemm}

\begin{proof}
We observe that $f(0)<f(\sqrt{2})<f(1)$, and hence $f$ reaches its maximum at an interior point.
Computing the first derivative we get 
\be\label{prima}
6\xi^2f'(\xi)=-\left(-\frac{4\xi(\xi^2-1)}{\xi^2-4} + \ln \left(\frac{2 + \xi}{2 - \xi}\right) + \xi^3 \ln\left(\frac{\xi^2}{4-\xi^2}\right)\right).
\ee
A direct computation shows that at the right hand side vanishes of the previous identity vanishes both at $\xi=0$ and at $\xi=1$. 
Computing $(6\xi^2 f')'$ we get
\be\label{seconda}
(6\xi^2 f')'(\xi)=3\xi^2\left(\frac{4(\xi^2-6)}{(\xi^2-4)^2}-\ln\left(\frac{\xi^2}{4-\xi^2}\right)\right)
\ee
The function $t\mapsto -\ln(t^{-2}(4-t^2))$ is strictly increasing, vanishes at $t=\sqrt{2}$ and decays to $-\infty$ when $t$ approaches $0$, while the function $t\mapsto -4(t^2-4)^{-1}(t^2-6)$
is strictly decreasing and equal to $-3/2$ when $t=0$. Hence there exists a unique $\xi_0\in (0,\sqrt{2})$ such that
the right hand side of \eqref{seconda} is positive for $\xi<\xi_0$ and negative otherwise. 
Hence the right hand side of \eqref{prima} can vanishes at most at a single interior point and the lemma is proved.
\end{proof}

\noindent{\bf Acknowledgments.} This paper was elaborated during a research period of LB at the  Universit\'e Savoie Mont Blanc. The Universit\'e Savoie Mont Blanc and its facilities are kindly acknowledged.
 LB is member of the Gruppo Nazionale per l'Analisi Matematica, la Probabilit\`a e le loro Applicazioni (GNAMPA) of the Istituto Nazionale di Alta Matematica (INdAM).

%%%%%%%%%%%%%%%%%%%%%%%%%%%%%%%%%%%%%%%%%%%%%%%%%%

\end{document}